\documentclass{amsart}

\usepackage{amssymb,latexsym,amsfonts,amsmath}
\usepackage{graphicx}
\include{diagrams}
\usepackage{eso-pic}
\usepackage{graphicx}
\usepackage{color}
\usepackage{type1cm}
\usepackage{graphicx}
\usepackage{cleveref}
\usepackage{enumitem}
\usepackage{comment}
\usepackage{xcolor}
\usepackage{subcaption}
\usepackage{bbm}

\DeclareMathOperator{\4all}{\text{ for~all~}}
\DeclareMathOperator{\R}{\mathbb{R}}
\DeclareMathOperator{\Rno}{\mathbb{R}^n_{\geq 0}}
\DeclareMathOperator{\twon}{2^{[n]}}
\DeclareMathOperator{\bbeta}{\boldsymbol\beta}
\DeclareMathOperator{\balpha}{\boldsymbol\alpha}

\DeclareMathOperator{\latone}{\mathcal{X}}
\DeclareMathOperator{\leqone}{\preceq}
\DeclareMathOperator{\geqone}{\succeq}

\DeclareMathOperator{\joinone}{\curlyvee}
\DeclareMathOperator{\meetone}{\curlywedge}

\DeclareMathOperator{\lattwo}{\mathcal{Y}}
\DeclareMathOperator{\leqtwo}{\sqsubseteq}
\DeclareMathOperator{\jointwo}{\sqcup}
\DeclareMathOperator{\meettwo}{\sqcap}

\newcommand{\tr}[1]{\mathrm{tr}\left(#1\right)}
\newcommand{\subjectto}{\text{subject to}}
\newcommand{\diag}[1]{\mathrm{diag}\left(#1\right)}

\newcommand{\minimize}[1]{\underset{#1}{\text{minimize}}~}
\newcommand{\maximize}[1]{\underset{#1}{\text{maximize}}~}
\newcommand{\supp}[1]{\mathrm{supp}\left(#1\right)}
\newcommand{\argmin}[1]{\underset{#1}{\text{argmin}}~}

\newtheorem*{assumptions}{Assumptions}
\makeatother

\usepackage{amssymb,latexsym,amsfonts,amsmath}
\usepackage{graphicx}

\newtheorem{theorem}{Theorem}[section]
\newtheorem{lemma}[theorem]{Lemma}

\newtheorem{proposition}[theorem]{Proposition}
\newtheorem{corollary}[theorem]{Corollary}

\theoremstyle{definition}

\theoremstyle{remark}
\newtheorem{remark}[theorem]{Remark}
\numberwithin{equation}{section}

\usepackage[parfill]{parskip}	
\usepackage[margin=1in]{geometry}
\usepackage{amsmath}
\usepackage{bbm}
\usepackage{graphicx}
\usepackage{amssymb}
\usepackage{mathtools}
\usepackage{enumitem}
\usepackage{comment}
\usepackage{xcolor}
\usepackage[ruled,vlined]{algorithm2e}
\usepackage{subcaption}

\DeclareMathOperator{\lleq}{\ll}
\DeclareMathOperator{\ljoin}{\rotatebox[origin=c]{90}{$\ll$}\hspace{2px}}
\DeclareMathOperator{\lmeet}{\rotatebox[origin=c]{270}{$\ll$}\hspace{2px}}

\DeclareMathOperator{\lsjoin}{\rotatebox[origin=c]{90}{$\Subset$}\hspace{2px}}
\DeclareMathOperator{\lsmeet}{\rotatebox[origin=c]{270}{$\Subset$}\hspace{2px}}

\begin{document}
\begin{abstract}
In model selection problems for machine learning, the desire for a well-performing model with meaningful structure is typically expressed through a regularized optimization problem.  In many scenarios, however, the meaningful structure is specified in some discrete space, leading to difficult nonconvex optimization problems.  In this paper, we connect the model selection problem with structure-promoting regularizers to submodular function minimization with continuous and discrete arguments.  In particular, we leverage the theory of submodular functions to identify a class of these problems that can be solved exactly and efficiently with an agnostic combination of discrete and continuous optimization routines.  We show how simple continuous or discrete constraints can also be handled for certain problem classes, and extend these ideas to a robust optimization framework.  We also show how some problems outside of this class can be embedded into the class, further extending the class of problems our framework can accommodate.  Finally, we numerically validate our theoretical results with several proof-of-concept examples with synthetic and real-world data, comparing against state-of-the-art algorithms.
\end{abstract}

\title[Continuous and discrete model selection]{Joint continuous and discrete model selection via submodularity}
\thanks{This work was supported in part by the US Army Research Laboratory
(ARL) under Cooperative Agreement W911NF-17-2-0196.}

\author[Bunton and Tabuada]{Jonathan Bunton and Paulo Tabuada}
\address{Department of Electrical Engineering\\
University of California at Los Angeles,
Los Angeles, CA 90095}
\email{tabuada@ee.ucla.edu}
\urladdr{http://www.ee.ucla.edu/$\sim$tabuada}
\email{j.bunton@ucla.edu}

\maketitle
\section{Introduction}
In many machine learning tasks, we require a model that not only performs a specified task well, but also has some meaningful structure.  Models with meaningful structure can, for example, be easier to understand and implement.  The desire for both accuracy and meaningful structure is usually expressed in a regularized optimization problem:
\begin{align}
\minimize{\mathbf{x}\in \mathcal{X}}f(\mathbf{x}) + \lambda g(\mathbf{x}). \label{eq:regularized_problem}
\end{align}
In this problem, $\mathbf{x}$ is a choice of model parameters from a parameter space $\mathcal{X}$, $f:\mathcal{X}\rightarrow\mathbb{R}$ is a function that describes the misfit of the model with the selected parameters to the given task (e.g., empirical risk), $g:\mathcal{X}\rightarrow\mathbb{R}$ is a function that expresses the deviation of our selected model parameters from some desired structure, and $\lambda\in\mathbb{R}_{\geq 0}$ is a tradeoff parameter.

Problem \eqref{eq:regularized_problem} becomes difficult when the desired model structure is an inherently discrete property, but the model parameters are continuous values $\mathbf{x}$ from a continuum $\mathcal{X}$.  A prime example of this issue arises in feature selection for sparse regression, where we seek a linear predictor $\mathbf{x}^*\in \mathcal{X}\subseteq\mathbb{R}^n$ such that:
\begin{align}
\mathbf{x}^*\in\argmin{\mathbf{x}\in \mathcal{X}} \Vert \mathbf{Ax}-\mathbf{b}\Vert_2^2 + \lambda \Vert \mathbf{x}\Vert_0,\label{eq:sparse_regression}
\end{align}
for some $\mathbf{A}\in\mathbb{R}^{m\times n}$ and $\mathbf{b}\in\mathbb{R}^m$, with $\Vert \mathbf{x} \Vert_2$ the standard Euclidean norm on $\mathbb{R}^m$, and $\Vert \mathbf{x}\Vert_0$ the $\ell_0$ pseudo-norm that counts the number of nonzero entries in the predictor $\mathbf{x}$.  The desired structure, in this case, is a sparse predictor $\mathbf{x}\in \mathcal{X}$. Sparsity, however, only depends on the combinatorial choice of zero entries in the model parameters $\mathbf{x}$, whereas the model also requires a choice of continuous values for $\mathbf{x}\in\mathcal{X}$.

Problems with this mixed dependence on both continuous and discrete properties of the model parameters such as \eqref{eq:sparse_regression} are notoriously difficult, and even NP-Hard in general \cite{rauhut2010compressive}.  A typical workaround is to replace the function describing model structure, $g$ in problem \eqref{eq:regularized_problem}, with a continuous relaxation that is more amenable to optimization.  One of the more celebrated instances of this approach is the relaxation of the $\ell_0$ pseudo-norm in \eqref{eq:sparse_regression} to the convex $\ell_1$ norm $\Vert \mathbf{x}\Vert_1$, which instead sums the absolute values of the vector $\mathbf{x}$.  While this relaxation still encourages the intended structure, the minimizer for the relaxed problem does not necessarily correspond to the minimizer for the initially specified problem \cite{bach2012structured}.  Moreover, the well-known conditions for sparse recovery in regression problems, such as Restricted Isometry Properties \cite{candes2005decoding}, Null Space Properties \cite{rauhut2010compressive}, and Irrepresentability Conditions \cite{zhao2006model}, are not applicable to more general discrete functions $g$.

In contrast, in this work we identify conditions that allow us to directly solve the originally posed regularized model-fitting problem \eqref{eq:regularized_problem} exactly and efficiently.  To derive our new conditions, we leverage submodularity, a property of functions that defines a boundary between easy and hard optimization problems.  Our approach stands in stark contrast to existing methods, which either focus on submodularity in purely one domain \cite{bach2019submodular} or relies on restricted isometry or strong convexity constants that are NP-Hard to compute \cite{elenberg2018restricted,eloptimal}.

Traditionally, submodularity is defined for functions on bounded discrete sets, where arbitrary function minimization is NP-Hard.  When a function is submodular, however, it can be minimized exactly in polynomial time \cite{schrijver2003combinatorial}.  The definition of submodularity extends to continuous functions as well, and recently the associated optimization guarantees have also been extended \cite{bach2019submodular,bian2017non}.  In particular, if a continuous function is submodular, it can also be minimized exactly in polynomial time.

The natural next question--which is addressed in this work--to ask is if submodularity still defines a boundary between easy and hard \emph{mixed} optimization problems such as \eqref{eq:regularized_problem}, where the function $f$ in \eqref{eq:regularized_problem} is continuous, but the function $g$ has a discrete co-domain.  Our work explores this boundary and identifies sufficient conditions, based on the submodularity of both functions, under which the exact solution of problem \eqref{eq:regularized_problem} can be efficiently computed.

Exploiting submodularity in these mixed scenarios is not a new idea, given its utility in discrete optimization problems.  Notable uses include establishing approximation guarantees for greedy algorithms applied to sparsity-constrained optimization \cite{elenberg2018restricted}, or in producing tight convex relaxations for set-function descriptions of desired sparsity patterns \cite{bach2012structured}.

As highlighted above, \cite{bach2019submodular} shows that if a continuous function is submodular, it can be \emph{discretized} into a discrete submodular function, which can then be minimized exactly in polynomial time.  However, this discretization is only valid for compact subsets of continuous spaces and necessarily introduces discretization error into the produced solution. 

In a line of work similar to this one, authors in \cite{eloptimal} propose converting the mixed problem to a purely discrete one without discretizing.  They then advocate using a specific submodular set function minimization algorithm for solving the discrete problem, and give approximation guarantees under the assumption that the functions are nearly submodular.  Our proposed approach is similar, but our work instead focuses on finding conditions under which an \emph{arbitrary choice} (of potentially more efficient) algorithms produce \emph{exact} results, which leads to their choice as a special case.

The sufficient conditions we require may be violated in practice.  Traditionally, violations of submodularity are handled by suitably relaxing the definition with an additive or multiplicative constant and propagating the constant through a particular algorithm \cite{eloptimal,elenberg2018restricted}.  Alternatively, in this work we find a sub-class of optimization problems that we can always lift into problems that satisfy our assumptions.  Moreover, we prove that the solution of the lifted problem gives a near-optimal solution to the original.  Our lifting approach stands in stark contrast to existing methods, as it is algorithm-independent with a guarantee that is easy to compute rather than tied to a specific algorithm and dependent on constants that are NP-Hard to compute \cite{eloptimal,elenberg2018restricted}.

We make several technical contributions, namely:
\begin{enumerate}[label=(\roman*)]
\item We identify new sufficient conditions, based on submodularity, under which the regularized model selection problem \eqref{eq:regularized_problem} can be solved efficiently and exactly;
\item We extend this theory to accommodate simple continuous and discrete constraints on the model parameter for some problem classes;
\item We highlight the utility of exact solutions for robust optimization scenarios;
\item We show that problems violating our sufficient conditions can be lifted to problems that do satisfy them, and whose solutions correspond to optimal or near-optimal solutions of the original problem;
\item We numerically validate the correctness of our theory with examples from sparse regression and retail price optimization.
\end{enumerate}

\section{Submodular Functions on Lattices}
In this work, we consider optimization problems defined on two sets: an uncountably infinite set, typically $\mathbb{R}^n$ or a subset thereof referred to as a \emph{continuous set}, and a countable set, typically finite and referred to as a \emph{discrete set}.  Because we would like to efficiently solve optimization problems defined on both continuous and discrete sets, we study a  structure that can allow efficient optimization in both cases: submodularity.

Submodularity is typically defined as a property of set functions, which are functions that map any subset of a finite set $V$ to a real number, i.e., $f:2^V\rightarrow\mathbb{R}$.  More generally, however, submodularity is a property of functions on \emph{lattices} which can be continuous or discrete sets.

Let $\latone$ be a set equipped with a partial order of its elements, denoted by $\leqone$. For any two elements $\mathbf{x},\mathbf{x'}\in\latone$ we define their least upper bound, or \emph{join}, as:
\begin{align}
\mathbf{x}\joinone\mathbf{x'} &= \inf\{\mathbf{y}\in\latone~:~\mathbf{x}\leq \mathbf{y},~\mathbf{x'}\leq \mathbf{y}\}.\label{eq:join_def}
\end{align}
Dually, we define their greatest lower bound, or \emph{meet}, as:
\begin{align}
\mathbf{x}\meetone\mathbf{x'}&=\sup\left\{\mathbf{y}\in\latone~:~\mathbf{y}\leq\mathbf{x},~\mathbf{y}\leq\mathbf{x'}\right\}.\label{eq:meet_def}
\end{align}
If for any two elements $\mathbf{x},\mathbf{x'}\in\latone$, their join, $\mathbf{x}\joinone\mathbf{x'}$, and their meet, $\mathbf{x}\meetone\mathbf{x'}$, exist and are in $\latone$, then the set $\latone$ and its order define a \emph{lattice}.  We write the lattice and its partial order together as $(\latone,\leqone)$, but will often write just $\latone$ when the order is clear from context.  If a subset $\mathcal{S}\subseteq\latone$ is such that for any two of its elements $\mathbf{x},\mathbf{x'}\in \mathcal{S}$, both their join, $\mathbf{x}\joinone\mathbf{x'}$, and their meet, $\mathbf{x}\meetone\mathbf{x'}$, are in $\mathcal{S}$, the subset $\mathcal{S}$ is called a \emph{sublattice} of $\latone$ \cite{davey2002introduction}.

As an example, consider a finite set of elements $V$.  Then its power set, $2^V$ (the set of all its possible subsets), forms a lattice when ordered by set inclusion, $\subseteq$.  Under this order, the join of any two elements $X,X'\subseteq V$ is their set union, $X\cup X'\subseteq V$, and dually, their meet is their set intersection $X\cap X'\subseteq V$.

We can also endow continuous sets with partial orders that define lattices.  Recent work has brought attention to $\mathbb{R}^n$ equipped with the partial order $\leqone$, defined as:
\begin{align}
\mathbf{x}\leqone\mathbf{x'}\quad\Leftrightarrow\quad\mathbf{x}_i \leq \mathbf{x}_i'\quad\text{for all }i=1,2,...,n,\label{eq:rn_order}
\end{align}
where $\leq$ denotes the usual order on $\mathbb{R}$.

Under this order, the join and meet operation for any two elements $\mathbf{x},\mathbf{x'}\in\mathbb{R}^n$ are element-wise maximum and minimum, respectively, meaning:
\begin{align}
(\mathbf{x}\joinone\mathbf{x'})_i &= \max\{\mathbf{x}_i,\mathbf{x}_i'\},\4all i=1,2,...,n,\label{eq:rn_join} \\
(\mathbf{x}\meetone\mathbf{x'})_i &= \min\{\mathbf{x}_i,\mathbf{x}_i'\},\4all i=1,2,...,n.\label{eq:rn_meet}
\end{align}

Given a lattice $\latone$, consider a function $f:\latone\rightarrow\mathbb{R}$.  The function $f$ is \emph{submodular} on the lattice $\latone$ when the following inequality holds for all $\mathbf{x},\mathbf{x}'\in\latone$:
\begin{align}
f(\mathbf{x})+f(\mathbf{x'}) \geq f(\mathbf{x}\joinone\mathbf{x'}) + f(\mathbf{x}\meetone\mathbf{x'}).\label{eq:lat_fn_submodular}
\end{align}
The function $f$ is \emph{monotone} when it satisfies:
\begin{align}
\mathbf{x}\leqone\mathbf{x'}\quad \implies \quad f(\mathbf{x}) \leq f(\mathbf{x'}).\label{eq:lat_fn_monotone}
\end{align}

When working with the lattice $(2^V,\subseteq)$, the submodular inequality \eqref{eq:lat_fn_submodular} becomes:
\begin{align}
f(A) + f(B) \geq f(A\cup B) + f(A\cap B) \quad \4all A,B\subseteq V.\label{eq:set_fn_submodular}
\end{align}
Similarly, the monotonicity implication \eqref{eq:lat_fn_monotone} becomes:
\begin{align}
A\subseteq B \quad\implies\quad f(A) \leq f(B).\label{eq:set_fn_monotone}
\end{align}
Minimizing or maximizing an arbitrary set function is NP-Hard in general. If the set function is submodular, however, it can be exactly minimized and approximately maximized (up to a constant-factor approximation ratio) in polynomial time \cite{schrijver2003combinatorial,nemhauser1978analysis}.  The computational tractability of submodular optimization for set functions has a variety of applications in countless fields such as sparse regression, summarization, and sensor placement \cite{elenberg2018restricted,hui2011class,krause2006near}.

When working with the lattice $(\mathbb{R}^n,\leqone)$, a function $f:\mathbb{R}^n\rightarrow\mathbb{R}$ is submodular when:
\begin{align}
f(\mathbf{x}) + f(\mathbf{x'}) \geq f(\max\{\mathbf{x},\mathbf{x'}\}) + f(\min\{\mathbf{x},\mathbf{x'}\})\quad\4all \mathbf{x},\mathbf{x'}\in\mathbb{R}^n,
\end{align}
where the maximum and minimum operations are performed element-wise, as expressed in \eqref{eq:rn_join} and \eqref{eq:rn_meet}.  When $f$ is twice differentiable, submodularity on $\mathbb{R}^n$ is equivalent  (see \cite{topkis1998supermodularity,bach2019submodular}) to the condition:
\begin{align}
\frac{\partial^2f}{\partial \mathbf{x}_i\partial\mathbf{x}_j} &\leq 0 \quad \4all i\neq j.\label{eq:submodular_hessian}
\end{align}
Perhaps surprisingly, the guarantees associated with submodular set function optimization extend to functions that are submodular on $\mathbb{R}^n$.  In particular, submodular functions on $\mathbb{R}^n$ can be minimized over a bounded sublattice in polynomial time (see \cite{bach2019submodular}), and can be approximately maximized with constant-factor approximation ratios \cite{bian2016guaranteed,bian2017non}.

\section{Problem Formulation}
In this section, we bridge continuous and discrete submodular function minimization in one unified problem statement.  We do this by drawing inspiration from the field of structured sparsity, where the choice of zero entries in real-valued decision variables is viewed as a coupled discrete and continuous problem \cite{bach2013learning,bach2011shaping}.

To highlight the connection with structured sparsity problems, for $n\in\mathbb{Z}_{>0}$, we denote by $[n]$ the set  $\{1,2,...,n\}$, and by $2^{[n]}$ the set of all possible subsets of $[n]$.  Define the map $\mathrm{supp}:\mathbb{R}^n\rightarrow 2^{[n]}$ as:
\begin{align}
\supp{\mathbf{x}} &= \{i\in [n]\mid \mathbf{x}_i\neq 0\}.\label{eq:supp}
\end{align}
In words, $\mathrm{supp}$ returns the set of indices where the vector $\mathbf{x}$ is nonzero.  Consider arbitrary functions $f:\mathbb{R}^n\rightarrow\mathbb{R}$ and $g:2^{[n]}\rightarrow\mathbb{R}$.  Problems of the form:
\begin{align}
\minimize{\mathbf{x}\in\mathbb{R}^n}~f(\mathbf{x}) + g(\supp{\mathbf{x}}), \label{eq:cont_discrete_opt_problem}
\end{align}
often arise in structured sparse optimization, where the preferences in discrete selections (the zero entries of $\mathbf{x}$) are expressed through the function $g$.  As a special case, if we let $f(\mathbf{x}) = \Vert \mathbf{D}\mathbf{x}-\mathbf{b}\Vert_2^2$ with $\mathbf{D}\in\mathbb{R}^{m\times n}$ and $\mathbf{b}\in\mathbb{R}^m$ and define $g(A) = \vert A\vert$ as the cardinality of the set $A$, \eqref{eq:cont_discrete_opt_problem} becomes:
\begin{align}
\minimize{\mathbf{x}\in\mathbb{R}^n}~\Vert\mathbf{D}\mathbf{x}-\mathbf{b}\Vert_2^2 + \Vert\mathbf{x}\Vert_0, \tag{CS}\label{eq:compressed_sensing}
\end{align}
where $\Vert\cdot\Vert_0$ denotes the $\ell_0$ pseudo-norm.  The problem \eqref{eq:compressed_sensing} is a form of the well-studied compressed sensing problem, which is NP-Hard in general \cite{rauhut2010compressive}.

Generalizing the idea of making continuous decisions through the choice of $\mathbf{x}$ in \eqref{eq:cont_discrete_opt_problem}, and discrete decisions through the choice of the zero entries of $\mathbf{x}$, we consider two lattices, $(\latone,\leqone)$ and $(\lattwo,\leqtwo)$, related by a map $\eta:\latone\rightarrow\lattwo$.  We let $f:\latone\rightarrow\mathbb{R}$ be a function describing the cost of assignments of variables in $\latone$, and similarly let $g:\lattwo\rightarrow\mathbb{R}$ describe the associated cost of choices in $\lattwo$. Then, we seek the optimal point $\mathbf{x}^*\in\latone$ in the problem:
\begin{align}
\minimize{\mathbf{x}\in\latone}~f(\mathbf{x})+g(\eta(\mathbf{x})).\tag{P}\label{eq:lattice_opt_problem}
\end{align}
Although we will eventually let $\latone$ describe continuous choices and $\lattwo$ describe associated discrete ones, our theoretical results do not rely on the cardinality of the lattices $\latone$ and $\lattwo$.

Intuitively, problem \eqref{eq:lattice_opt_problem} asks for the element $\mathbf{x}\in\latone$ which incurs minimum cost in $\latone$, as measured by $f(\mathbf{x})$, and in $\lattwo$, as measured by $g(\eta(\mathbf{x}))$.  Given that the special case of \eqref{eq:compressed_sensing} is already hard in general, with no additional structure on $f$, $g$ and $\eta$, this problem is hopelessly difficult.  To provide the necessary structure, we make the following assumptions.
\begin{assumptions}
Consider the lattices $(\latone,\leqone)$ and $(\lattwo,\leqtwo)$ and the maps $\eta:\latone\rightarrow\lattwo$, $f:\latone\rightarrow\mathbb{R}$, and $g:\lattwo\rightarrow\mathbb{R}$.  We make the following assumptions:
\begin{enumerate}
\item The functions $f$ and $g$ are submodular on the lattices $\latone$ and $\lattwo$, respectively,
\item The function $g$ is monotone on $\lattwo$,
\item For all $\mathbf{x},\mathbf{x}'\in\latone$:
\begin{align*}
\eta(\mathbf{x}\joinone\mathbf{x}') \leqtwo \eta(\mathbf{x})\jointwo\eta(\mathbf{x}'), \quad \eta(\mathbf{x}\meetone\mathbf{x}')\leqtwo \eta(\mathbf{x})\meettwo\eta(\mathbf{x}').
\end{align*}
\end{enumerate}
\end{assumptions}

\begin{remark}
If the map $\eta:\latone\rightarrow\lattwo$ satisfies Assumption 3,  it is an order-preserving join-homomorphism, meaning it maintains the order and joins of elements in $\latone$.  (Prop. 2.19 in \cite{davey2002introduction})  Explicitly, Assumption 3 is equivalent to the condition that for any $\mathbf{x}, \mathbf{x}'\in\latone$:
\begin{gather*}
\mathbf{x}\leqone \mathbf{x}' \Rightarrow \eta(\mathbf{x})\leqtwo\eta(\mathbf{x}'),\\
\eta(\mathbf{x}\joinone\mathbf{x}') = \eta(\mathbf{x})\jointwo\eta(\mathbf{x}').
\end{gather*}  Despite this equivalence, we leave Assumption 3 as written above for clarity in future proofs.
\end{remark}
We highlighted the lattices $(\mathbb{R}^n,\leqone)$ and $(2^{[n]},\subseteq)$,  but for the map $\mathrm{supp}:\mathbb{R}^n\rightarrow 2^{[n]}$ to satisfy Assumption 3, we must restrict the domain of $f$ to only only the first orthant, $(\mathbb{R}^n_{\geq 0},\leqone)$.  As mentioned by \cite{bian2017non}, this issue can often be resolved by considering an appropriate \emph{orthant conic lattice}, which views $\mathbb{R}^n$ as a product of $n$ copies of $\mathbb{R}$ and selects a different order for each copy.  Alternatively, any least-squares problem such as \eqref{eq:compressed_sensing} can be lifted to a non-negative least-squares problem, allowing us to satisfy Assumption 3 with the map $\mathrm{supp}$, but potentially no longer satisfying Assumption 1 (see Appendix \ref{apdx:NNLS}).

Assumption 1, which requires $f$ and $g$ to be submodular can be restrictive in practice.  To mitigate this, in Section \ref{sec:quadratic_lifting} we show how some specific problem instances that do not satisfy Assumption 1--in particular when $f$ is quadratic--can be lifted to a new optimization problem that satisfies all the required assumptions.  We then derive conditions under which solving the new, lifted problem still provides a solution to the original problem that violated Assumption 1.  In contrast, the more typical way of handling non-submodular $f$ involves relaxing the definition of submodularity \eqref{eq:lat_fn_submodular} to include an additive or multiplicative constant and propagating it through a chosen algorithm to give near-optimality guarantees. \cite{eloptimal,elenberg2018restricted} Our suggested lifting, however, sidesteps the need for a particular algorithm while still providing optimality or near-optimality guarantees.

\section{Solving an Equivalent Problem}
In this section, we outline our approach for solving the problem \eqref{eq:lattice_opt_problem} by defining a related optimization problem on a single lattice.  We then prove that this related problem is a submodular function minimization problem, and that by solving it we recover a solution to \eqref{eq:lattice_opt_problem}.  Finally, we highlight some conditions under which solving this related problem is a polynomial time operation.

\subsection{The Equivalent Submodular Minimization Problem}
As expressed above, the problem \eqref{eq:lattice_opt_problem} asks for the a choice of $\mathbf{x}\in\latone$ and associated $\eta(\mathbf{x})\in\lattwo$.  Our key observation is that we could instead ask for a choice of $\mathbf{y}\in\lattwo$ and best associated $\mathbf{x}\in\latone$, leading to the problem:
\begin{align*}
\minimize{\mathbf{y}\in\lattwo}~g(\mathbf{y}) + \underset{\substack{\mathbf{x}\in\latone\\ \eta(\mathbf{x}) = \mathbf{y}}}{\min}~f(\mathbf{x}).
\end{align*}
In the special case of \eqref{eq:compressed_sensing} explored earlier, this equivalent problem becomes:
\begin{align*}
\minimize{S\in 2^{[n]}}~\vert S\vert + \underset{\substack{\mathbf{x}\in\mathbb{R}^n_{\geq 0}\\ \supp{\mathbf{x}} = S}}{\min}~\Vert\mathbf{A}\mathbf{x}-\mathbf{b}\Vert_2^2.
\end{align*}
While this new problem is clearly the same as \eqref{eq:compressed_sensing}, the innermost minimization is over the set of $\mathbf{x}\in\mathbb{R}^n_{\geq 0}$ such that $\supp{\mathbf{x}} = S$, or equivalently, $\mathbf{x}_i \neq 0$ for all $i\in S$, and $\mathbf{x}_i = 0$ for all $i\notin S$.  This feasible set is not a closed subset of $\mathbb{R}^n_{\geq 0}$, and thus the corresponding minimizer of this innermost problem may not exist \cite{borwein2000convex}.

With this issue in mind, we instead consider a slight relaxation of the above problem:
\begin{align*}
\minimize{\mathbf{y}\in\lattwo}~g(\mathbf{y}) + H(\mathbf{y}), \tag{P-R}\label{eq:lattice_opt_prob_relax}
\end{align*}
where we have defined the function $H:\lattwo\rightarrow\mathbb{R}$ as:
\begin{align}
H(\mathbf{y}) &= \underset{\substack{\mathbf{x}\in\latone\\ \eta(\mathbf{x}) \leqtwo \mathbf{y}}}{\min}~f(\mathbf{x}).\label{eq:H_definition}
\end{align}
In the special case of \eqref{eq:compressed_sensing}, this relaxation produces the problem:
\begin{align}\label{eq:compressed_sensing_relax}
\minimize{S\in 2^{[n]}}~\vert S\vert + \underset{\substack{\mathbf{x}\in\mathbb{R}^n_{\geq 0}\\ \supp{\mathbf{x}} \subseteq S}}{\min}~\Vert\mathbf{A}\mathbf{x}-\mathbf{b}\Vert_2^2,\tag{CS-R}
\end{align}
where the innermost minimization is instead over the set of $\mathbf{x}\in\mathbb{R}^n_{\geq 0}$ such that $\mathbf{x}_i = 0$ for all $i\notin S$, which is a closed subset of $\mathbb{R}^n_{\geq 0}$.

We now prove that under Assumptions 1-3, the relaxed problem \eqref{eq:lattice_opt_prob_relax} is a submodular minimization problem, and that by solving it we can recover the corresponding minimizer for \eqref{eq:lattice_opt_problem}.  As established above, minimizing functions on finitely presentable distributive lattices is efficient when the functions are submodular, so we show that the relaxed problem \eqref{eq:lattice_opt_prob_relax} is a submodular function minimization problem on $\lattwo$.

\begin{theorem}\label{thm:main_result}
Under Assumptions 1-3, the function $g + H:\lattwo\rightarrow\mathbb{R}$ is submodular on $\lattwo$, and therefore the relaxed problem \eqref{eq:lattice_opt_prob_relax} is a submodular function minimization problem over $\lattwo$.  Moreover, let $\mathbf{y}^*\in\lattwo$ be the minimizer for the problem \eqref{eq:lattice_opt_prob_relax}, and let $\mathbf{x}^*\in\latone$ be such that:
\begin{align*}
\mathbf{x}^*\in\underset{\substack{\mathbf{x}\in\latone \\ \eta(\mathbf{x})\leqtwo\mathbf{y}^*}}{\mathrm{argmin}}~f(\mathbf{x}).
\end{align*}
Then $\mathbf{x}^*$ is a minimizer for the problem \eqref{eq:lattice_opt_problem}.
\end{theorem}

To prove this result, we require a few technical lemmas.

\begin{lemma}\label{lem:sublattice}
Let $(\latone,\leqone)$ and $(\lattwo,\leqtwo)$ be lattices with the map $\eta:\latone\rightarrow\lattwo$ satisfying Assumption 3.  Then the set:
\begin{align}
\mathcal{D} &= \left\lbrace(\mathbf{x},\mathbf{y})\in \latone\times\lattwo \mid\eta(\mathbf{x})\leqtwo \mathbf{y}\right\rbrace, \label{eq:sublattice_D}
\end{align}
is a sublattice of the product lattice, $\latone\times \lattwo$.
\end{lemma}
\begin{proof}
On the product lattice, the join of any two elements $(\mathbf{x},\mathbf{y}),(\mathbf{x}',\mathbf{y}')\in \mathcal{D}$ is denoted by $\vee_{\mathcal{D}}$, and defined as:
\begin{align*}
(\mathbf{x},\mathbf{y})\vee_{\mathcal{D}}(\mathbf{x}',\mathbf{y}') &= (\mathbf{x}\joinone\mathbf{x}',\mathbf{y}\jointwo\mathbf{y}').
\end{align*}
Then, we note that for this same $(\mathbf{x},\mathbf{y}),(\mathbf{x}',\mathbf{y}')\in \mathcal{D}$:
\begin{align*}
\eta(\mathbf{x}\joinone\mathbf{x}')&\leqtwo\eta(\mathbf{x})\jointwo\eta(\mathbf{x}') \leqtwo \mathbf{y}\jointwo\mathbf{y}',
\end{align*}
where we first used Assumption 3, then the fact that $(\mathbf{x},\mathbf{y}),(\mathbf{x}',\mathbf{y}')\in \mathcal{D}$.  Therefore, the pair $(\mathbf{x}\joinone\mathbf{x}',\mathbf{y}\jointwo\mathbf{y}')$ is also in $D$.

Because $(\mathbf{x},\mathbf{y})$ and $(\mathbf{x}',\mathbf{y}')$ were arbitrary, this holds for all of $\mathcal{D}$.  A dual analysis follows for the meet operation.
\end{proof}
The sublattice $\mathcal{D}$ is useful as the only pairs of $(\mathbf{x},\mathbf{y})\in\latone\times\lattwo$ considered in the problem \eqref{eq:lattice_opt_prob_relax} are those that are in $\mathcal{D}$.  The following theorem then uses this sublattice to prove that $H$ is submodular.  The result is a simple application of an established theorem in literature, but we include its proof here for completeness.\\
\begin{theorem}\label{thm:topkis}
(Application of \textit{Theorem 2.7.6 in \cite{topkis1998supermodularity}}) Let $f:\latone\rightarrow\mathbb{R}$, $g:\lattwo\rightarrow\mathbb{R}$, and $\eta:\latone\rightarrow\lattwo$ be maps satisfying Assumptions 1 and 3.  Then the function $g + H:\lattwo\rightarrow\mathbb{R}$, with $H$ defined as in \eqref{eq:H_definition}, is submodular on $\lattwo$.
\end{theorem}
\begin{proof}
To prove this statement, we take two points $\mathbf{y},\mathbf{y}'\in\lattwo$ and compare the values of the function $g + H$, verifying the submodular inequality \eqref{eq:lat_fn_submodular}.  We note that for any $\mathbf{y},\mathbf{y'}\in\lattwo$, there are corresponding $\mathbf{z},\mathbf{z}'\in\latone$ such that:
\begin{align}
\begin{aligned}
\mathbf{z}&\in\argmin{\substack{\mathbf{x}\in\latone \\ \eta(\mathbf{x})\leqtwo \mathbf{y}}}~f(\mathbf{x}) \quad \Rightarrow\quad H(\mathbf{y}) = f(\mathbf{z}),\\
\mathbf{z}'&\in\argmin{\substack{\mathbf{x}\in\latone \\ \eta(\mathbf{x})\leqtwo \mathbf{y}'}}~f(\mathbf{x})\quad \Rightarrow \quad H(\mathbf{y}') = f(\mathbf{z}').
\end{aligned}\label{eq:z_optimality}
\end{align}
By definition, $(\mathbf{z},\mathbf{y})$ and $(\mathbf{z}',\mathbf{y}')$ are both in the subset $\mathcal{D}$ as defined in \eqref{eq:sublattice_D}.  Then, it follows:
\begin{align*}
g(\mathbf{y}) + H(\mathbf{y}) + g(\mathbf{y}') + H(\mathbf{y}') &= g(\mathbf{y}) + f(\mathbf{z}) + g(\mathbf{y}') + f(\mathbf{z}') \\
&\geq g(\mathbf{y}\jointwo\mathbf{y'}) + g(\mathbf{y}\meettwo\mathbf{y}') + f(\mathbf{z}\joinone\mathbf{z}') + f(\mathbf{z}\meetone\mathbf{z}'),
\end{align*}
where we first used \eqref{eq:z_optimality} and then the submodularity of $f$ and $g$.

By Lemma \ref{lem:sublattice}, $\mathcal{D}$ is a sublattice of $\latone\times\lattwo$, and so the pairs $(\mathbf{z}\joinone\mathbf{z}',\mathbf{y}\jointwo\mathbf{y}')$ and $(\mathbf{z}\meetone\mathbf{z}',\mathbf{y}\meettwo\mathbf{y}')$ are also in $\mathcal{D}$, meaning:
\begin{align*}
\eta(\mathbf{z}\joinone\mathbf{z}') &\leqtwo \mathbf{y}\jointwo\mathbf{y}', \\ \eta(\mathbf{z}\meetone\mathbf{z}')&\leqtwo\mathbf{y}\meettwo\mathbf{y}'.
\end{align*}
Therefore $\mathbf{z}\joinone\mathbf{z}'$ and $\mathbf{x}\meetone\mathbf{x}'$ are feasible points in the minimization defining $H(\mathbf{y}\jointwo\mathbf{y}')$ and $H(\mathbf{y}\meettwo\mathbf{y}')$, respectively, in \eqref{eq:H_definition}.  We then have, as desired:
\begin{align*}
g(\mathbf{y}) + H(\mathbf{y}) + g(\mathbf{y}') + H(\mathbf{y}') &\geq g(\mathbf{y}\jointwo\mathbf{y'}) + g(\mathbf{y}\meettwo\mathbf{y}') + f(\mathbf{z}\joinone\mathbf{z}') + f(\mathbf{z}\meetone\mathbf{z}') \\
&\geq g(\mathbf{y}\jointwo\mathbf{y'}) + g(\mathbf{y}\meettwo\mathbf{y}') +  \underset{\substack{\mathbf{x}\in\latone\\ \eta(\mathbf{x}) \leqtwo \mathbf{y}\jointwo\mathbf{y}'}}{\min}\hspace{-2.5mm}f(\mathbf{x}) +  \underset{\substack{\mathbf{x}\in\latone\\ \eta(\mathbf{x}) \leqtwo \mathbf{y}\meettwo\mathbf{y}'}}{\min}\hspace{-2.5mm}f(\mathbf{x})\\
&= g(\mathbf{y}\jointwo\mathbf{y'}) + H(\mathbf{y}\jointwo\mathbf{y}') + g(\mathbf{y}\meettwo\mathbf{y}') + H(\mathbf{y}\meettwo\mathbf{y}').
\end{align*}
\end{proof}
Because $g+H$ is submodular on $\lattwo$, solving \eqref{eq:lattice_opt_prob_relax}, is an instance of submodular function minimization.  What remains is to show that solving this relaxed problem allows us to also solve to the original problem, \eqref{eq:lattice_opt_problem}.

\begin{lemma}\label{lem:minimizers}
Let $\mathbf{y}^*\in\lattwo$ be a minimizer for the relaxed problem \eqref{eq:lattice_opt_prob_relax}, and let $\mathbf{x}^*\in\latone$ be such that:
\begin{align*}
\mathbf{x}^*\in\underset{\substack{\mathbf{x}\in\latone\\\eta(\mathbf{x})\leqtwo\mathbf{y}^*}}{\mathrm{argmin}}~f(\mathbf{x}).
\end{align*}
If $g$ satisfies Assumption 2, then $\mathbf{x}^*$ is a minimizer for the problem \eqref{eq:lattice_opt_problem}.
\end{lemma}
\begin{proof}
To prove this lemma, we consider an optimal $\mathbf{z}^*\in\latone$ for problem \eqref{eq:lattice_opt_problem} and verify that the proposed minimizer, $\mathbf{x}^*\in\latone$, has the same cost.

We first note that by the optimality of $\mathbf{z}^*$ in problem \eqref{eq:lattice_opt_problem}:
\begin{align}
f(\mathbf{z}^*) + g(\eta(\mathbf{z}^*)) \leq f(\mathbf{x}^*) + g(\eta(\mathbf{x}^*)). \label{eq:optimality}
\end{align}

Additionally, we have:
\begin{align*}
\begin{array}{clc}
f(\mathbf{z}^*) + g(\eta(\mathbf{z}^*)) &\geq \underset{\substack{\mathbf{x}\in\latone\\\eta(\mathbf{x})\leqtwo\eta(\mathbf{z}^*)}}{\min}~f(\mathbf{x}) + g(\eta(\mathbf{z}^*))&\quad\text{(minimizing, as $\mathbf{z}^*$ is feasible)} \\
&= H(\eta(\mathbf{z}^*)) + g(\eta(\mathbf{z}^*))&\quad\text{(definition of $H$)} \\
&\geq H(\mathbf{y}^*) + g(\mathbf{y}^*)&\quad\text{(optimality of $\mathbf{y}^*$ in \ref{eq:lattice_opt_prob_relax})}\\
&= f(\mathbf{x}^*) + g(\mathbf{y}^*)&\quad\text{(definition of $\mathbf{x}^*$).}
\end{array}
\end{align*}
This sequence of inequalities implies:
\begin{align}
f(\mathbf{z}^*) + g(\eta(\mathbf{z}^*)) & \geq f(\mathbf{x}^*) + g(\mathbf{y}^*). \label{eq:lemma_two_ineq}
\end{align}
Note that because $g$ is monotone, $g(\mathbf{y}^*) \geq g(\eta(\mathbf{x}^*))$.  Using this fact, we can lower bound the right-hand side of \eqref{eq:lemma_two_ineq}:
\begin{align*}
f(\mathbf{z}^*) + g(\eta(\mathbf{z}^*)) &\geq f(\mathbf{x}^*) + g(\mathbf{y}^*)\geq f(\mathbf{x}^*) + g(\eta(\mathbf{x}^*)).
\end{align*}
By the optimality of $\mathbf{z}^*$, we see that $\mathbf{x}^*$ must also be optimal for the problem \eqref{eq:lattice_opt_problem}.
\end{proof}
This series of results gives rise to Theorem \ref{thm:main_result}, which provides sufficient conditions under which we can transform problem \eqref{eq:lattice_opt_problem}, an optimization problem on two lattices, into problem \eqref{eq:lattice_opt_prob_relax}, a submodular function minimization problem on a single lattice.

\begin{proof} \textit{(Theorem \ref{thm:main_result})}\\ Under Assumptions 1 and 3, Theorem \ref{thm:topkis} states that the function $g + H:\lattwo\rightarrow\mathbb{R}$ is submodular on the lattice $\lattwo$.  Therefore, solving \eqref{eq:lattice_opt_prob_relax} is a submodular function minimization problem over $\lattwo$, and the first part of the theorem is proved.

Under Assumption 2, by Lemma \ref{lem:minimizers}, given the minimizer $\mathbf{y}^*$ of \eqref{eq:lattice_opt_prob_relax}, the point $\mathbf{x}^*\in\latone$ defined by:
\begin{align*}
\mathbf{x}^*\in\underset{\substack{\mathbf{x}\in\latone\\\eta(\mathbf{x})\leqtwo\mathbf{y}^*}}{\mathrm{argmin}}~f(\mathbf{x}),
\end{align*} is a minimizer in the original problem \eqref{eq:lattice_opt_problem}.
\end{proof}
\subsection{Solving (P-R) in Polynomial Time\label{sec:polytime}}
Despite the submodular structure of the functions, we can only truly solve \eqref{eq:lattice_opt_prob_relax} in polynomial time if $\lattwo$ is a finitely presentable distributive lattice and we have an oracle for evaluating the functions $g$ and $H$, which we formally state next.

\begin{corollary}\label{cor:polytime}
Let $f:\latone\to\R$ be a submodular function on $(\latone,\leqone)$, let $(\lattwo,\leqtwo)$ be a finitely presentable distributive or diamond modular lattice with $g:\lattwo\to\R$ a monotone submodular function, and let $\eta:\latone\to\lattwo$ satisfy Assumption 3.  If we have access to an evaluation oracle for $g + H$, then problem \eqref{eq:lattice_opt_problem} can be solved in a polynomial number of operations and a polynomial number of calls to the oracle.
\end{corollary}
\begin{proof}
Assumptions 1, 2, and 3 are satisfied, by $\latone$, $\lattwo$, and the functions $\eta$, $f$, and $g$.  By Theorem \ref{thm:main_result}, therefore, we can solve the problem \eqref{eq:lattice_opt_problem} by instead minimizing $g+H$ over $\lattwo$, i.e., solving problem \eqref{eq:lattice_opt_prob_relax}.  Problem \eqref{eq:lattice_opt_prob_relax} is a submodular function minimization problem over a a finitely presentable distributive or diamond modular lattice, which established algorithms can solve in a polynomial number of operations and oracle calls to $g+H$ \cite{fujishige2022minimizing,schrijver2003combinatorial}.
\end{proof}
With Corollary \ref{cor:polytime} in hand, we need to construct the required oracle for $H:\lattwo\to\R$ that only requires a polynomial number of operations.  Once we have an oracle for $H$ (assuming another oracle or polynomial algorithm for evaluating $g$), solving \eqref{eq:lattice_opt_problem} clearly only requires a polynomial number of operations.

We are particularly interested in joint continuous and discrete optimization, such as when the relevant lattices are $(\latone,\leqone) = (\mathbb{R}^n_{\geq 0},\leqtwo)$ and $(\lattwo,\leqone) = (2^{[n]},\subseteq)$ connected by the map \mbox{$\mathrm{supp}:\mathbb{R}^n_{\geq 0}\rightarrow 2^{[n]}$} as expressed in \eqref{eq:supp}.  In this case, evaluating $H$ requires solving the optimization problem:
\begin{align}
\minimize{\substack{\mathbf{x}\in\mathbb{R}^n_{\geq 0}\\\supp{\mathbf{x}}\subseteq A}} f(\mathbf{x}),\label{eq:h_cont_discrete}
\end{align}
for any $A\in 2^{[n]}$.

As discussed above, when $\latone$ is the product of bounded intervals, we can rely on the continuous submodular minimization algorithms developed by \cite{bach2019submodular}.  These algorithms, however, introduce discretization error, limiting the accuracy of the evaluations of $H$.  Moreover, the simple example of \eqref{eq:h_cont_discrete} is a continuous submodular minimization problem, but the set $\mathbb{R}_{\geq 0}$ is not a bounded sublattice and thus the algorithms of \cite{bach2019submodular} do not directly apply.
 Continuous submodularity alone appears limited in this way, so we pursue other problem structures leading to algorithms for efficient and arbitrarily accurate solutions of \eqref{eq:h_cont_discrete}.

Note that for any $A\in 2^{[n]}$, the feasible set for the sub-problem \eqref{eq:h_cont_discrete} is a convex subset of $\mathbb{R}^n_{\geq 0}$.  If the function $f:\mathbb{R}^n_{\geq 0}\rightarrow\mathbb{R}$ is convex, under appropriate regularity conditions, we can use any generic convex optimization routine to solve the associated sparsity-constrained problem \eqref{eq:h_cont_discrete}.  For example, in the compressed sensing scenario shown  in \eqref{eq:compressed_sensing_relax}, evaluating $H$ amounts to solving a simple reduced least-squares problem.  More generally, we need $f:\latone\to\R$ to be convex and submodular, and the set of $\mathbf{x}\in\latone$ such that $\eta(\mathbf{x})\leqtwo\mathbf{y}$ to be a compact, convex subset for every $\mathbf{y}\in\lattwo$, alongside sufficient regularity conditions, such as constraint qualifications or the existence of separation oracles \cite{borwein2000convex,schrijver2003combinatorial}.

We have already assumed that $f$ is submodular (in this case, on $\mathbb{R}^n_{\geq 0}$), but submodular functions are neither a subset nor a superset of convex functions, so we may also require that $f$ is convex.  For example, any separable convex function $f$ satisfies this assumption, as do convex quadratic functions with non-positive off-diagonal entries, or functions on $\mathbb{R}^n$ that can be identified as the Lov\'{a}sz extension of submodular \emph{set} functions.  

Our theory is completely agnostic to the choice of algorithms for both evaluating $H$ and solving the discrete optimization problem \eqref{eq:lattice_opt_prob_relax}.  In particular, if we assume $f$ is convex, evaluate it through convex optimization, and use projected subgradient descent on the Lov\`asz extension of $g+H$ as the algorithm for solving the set function minimization, we recover exactly the approach proposed by \cite{eloptimal}.

Convexity of $f$ is not the only additional assumption on $f$ that leads to tractable evaluations of $H$ without resorting to continuous submodular minimization algorithms. As an alternative, we could consider a nonconvex quadratic form for $f:\mathbb{R}^n_{\geq 0}\rightarrow\mathbb{R}$:
\begin{align}
f(\mathbf{x}) &= \mathbf{x}^T\mathbf{Q}\mathbf{x} + \mathbf{p}^T\mathbf{x},\label{eq:quadratic_f}
\end{align}
with $\mathbf{Q}\in\mathbb{R}^{n\times n}$ and $\mathbf{p}\in\mathbb{R}^n$.  The assumption that this quadratic function is submodular on $\mathbb{R}^n_{\geq 0}$ is equivalent to the condition:
\begin{align*}
\frac{\partial^2 f}{\partial \mathbf{x}_i\partial\mathbf{x}_j} &= \mathbf{Q}_{ij} \leq 0, \quad \text{for all }i\neq j.
\end{align*}
Moreover, for a given $A\in 2^{[n]}$,  our sub-problem instance \eqref{eq:h_cont_discrete} is a constrained, nonconvex quadratic program:
\begin{align}
\begin{array}{cc}
\minimize{\mathbf{x}\in\mathbb{R}^n} & \mathbf{x}^T\mathbf{Q}\mathbf{x} + 2\mathbf{p}^T\mathbf{x} \\
\text{subject to}& \mathbf{x} \geq 0 \\
& \mathbf{x}_i = 0, ~i\notin A.
\end{array}\label{eq:quadratic_h}
\end{align}
Researchers \cite{kim2003exact} have established that nonconvex quadratic programs satisfying submodularity admit tight semidefinite program relaxations. In particular, we have the following theorem:
\begin{theorem}
(\textit{Theorem 3.1 in \cite{kim2003exact}})  Let $\mathbf{Q}\in\mathbb{R}^{n\times n}$ have nonpositive off-diagonal entries. Let $\mathrm{tr}:\mathbb{R}^{n\times n}\rightarrow\mathbb{R}$ denote the trace of a matrix, $\mathrm{diag}:\mathbb{R}^{n\times n}\rightarrow\mathbb{R}^n$ denote the diagonal entries of the matrix, and let $\succeq$ indicate the positive semidefiniteness of a symmetric matrix. Further, for any $A\in 2^{[n]}$, let $\mathbf{Z}_{A^c}$ denote the rows and columns of $\mathbf{Z}$ with indices not in the set $A$. Consider the semi-definite program:
\begin{align*}
\begin{array}{cc}
\minimize{\substack{\mathbf{z}\in\mathbb{R}^n \\ \mathbf{Z}\in \mathbb{S}^n}} & \tr{\mathbf{QZ}} + 2\mathbf{p}^T\mathbf{z} \\
\subjectto & \tr{\mathbf{Z}_{A^c}} \leq 0 \\
& \diag{\mathbf{Z}} \geq 0 \\
& \begin{bmatrix}
1 & \mathbf{z}^T \\
\mathbf{z} & \mathbf{Z}
\end{bmatrix} \succeq 0,
\end{array}
\end{align*}
Given the solution $(\mathbf{Z}^*,\mathbf{z}^*)$ to this SDP, the vector $\mathbf{x}^*_i = \sqrt{\mathbf{Z}^*_{ii}}$, $i=1,...,n$ is a minimizer for the non-convex quadratic program \eqref{eq:quadratic_h}.
\end{theorem}

Because semi-definite programs satisfying appropriate constraint qualifications can be solved in polynomial time, we could use this relaxation to evaluate $H$ for any subset $A\in 2^{[n]}$.  This approach produces the required oracle for Corollary \ref{cor:polytime}, but only requires that quadratic functions $f$ of the form \eqref{eq:quadratic_f} satisfy submodularity.


\section{Constrained Optimization}
In this and the following sections, we extend our framework both theoretically and algorithmically for the specific case of the lattices $(\mathbb{R}^n_{\geq 0}, \leqone)$ and $(2^{[n]},\subseteq)$, connected by the support map $\mathrm{supp}:\mathbb{R}^n_{\geq 0}\rightarrow 2^{[n]}$.

In many problems, we may be interested in optimization over a feasible strict subset $C\subset\mathbb{R}^n_{\geq 0}$.  Unfortunately, submodular function minimization and maximization subject to constraints is NP-Hard in general \cite{fujishige2011submodular}.  This difficulty arises because arbitrary subsets of a lattice rarely define sublattices.

One simple class of problems whose feasible sets are not sublattices are problems with \emph{budget constraints}:
\begin{align}
\begin{array}{cc}
\minimize{\mathbf{x}\in\mathbb{R}^n_{\geq 0}} & f(\mathbf{x}) + g(\supp{\mathbf{x}}) \\
\subjectto & \sum_{i=1}^nW_i(\mathbf{x}_i) \leq B,
\end{array}\label{eq:constrained_case}
\end{align}
with $W_i:\mathbb{R}_{\geq 0}\rightarrow\mathbb{R}$ strictly increasing functions for $i = 1,2,...,n$ and $B \in \mathbb{R}_{>0}$ a ``budget''.

When confronted with constrained optimization problems such as \eqref{eq:constrained_case}, one common approach is to add a Lagrange multiplier $\mu\in\mathbb{R}_{\geq 0}$ and instead solve the unconstrained problem:
\begin{align}
\minimize{\mathbf{x}\in\mathbb{R}^n_{\geq 0}}&~f(\mathbf{x}) + g(\supp{\mathbf{x}}) + \mu \sum_{i=1}^nW_i(\mathbf{x}_i).\label{eq:regularized_case}
\end{align}
For the correct choice of $\mu\in\mathbb{R}_{\geq 0}$, solving the regularized problem \eqref{eq:regularized_case} can be equivalent to solving the constrained problem \eqref{eq:constrained_case} \cite{nagano2011size,staib2019robust}.  Because \eqref{eq:constrained_case} is non-convex, identifying when this approach is valid requires some careful detail.  When possible, however, determining the $\mu$ that renders the two problems equivalent is typically a difficult task.

Our work in this section relies on the following result that relates parameterized families of submodular set function minimization problems to a single convex optimization problem.

\begin{theorem}(Proposition 8.4 in \cite{bach2013learning})\label{thm:bach_convex_submod} Let $h:2^{[n]}\rightarrow\mathbb{R}$ be a submodular set function, and $h_L:\mathbb{R}^n\rightarrow\mathbb{R}$ its Lov\`asz extension (which is therefore convex).  If, for some $\epsilon > 0$, $\psi_i:\mathbb{R}_{\geq \epsilon}\rightarrow\mathbb{R}$ is a strictly increasing function on its domain for all $i=1,2,...,n$, then the minimizer $\mathbf{u}^*\in\mathbb{R}^n_{\geq 0}$ of the convex optimization problem:
\begin{align}\label{eq:single_convex}
\minimize{\mathbf{u}\in\mathbb{R}^n_{\geq 0}}~h_L(\mathbf{u})+\sum_{i=1}^n\int_{\epsilon}^{\epsilon+\mathbf{u}_i}\psi_i(\mu)d\mu,
\end{align}
is such that the set $A^\mu = \{i\in [n]~:~\mathbf{u}_i^* > \mu\}$ is the minimizer with smallest cardinality for the submodular set function minimization problem:
\begin{align}\label{eq:family_submodular}
\minimize{A\in 2^{[n]}}~h(A) + \sum_{i\in A}\psi_i(\mu),
\end{align}
for any $\mu\in\mathbb{R}_{\geq \epsilon}$.
\end{theorem}

In the following subsections we identify classes of problems that allow the regularized problem \eqref{eq:regularized_case} to be expressed in the form given by \eqref{eq:family_submodular}.  Theorem \ref{thm:bach_convex_submod} then provides a single convex optimization problem we can solve to recover the solution to \eqref{eq:regularized_case} for all possible values of the regularization strength $\mu$.  In prior work, this same theory was applied to purely discrete submodular minimization problems \cite{fujishige2011submodular}, and purely continuous submodular minimization problems \cite{staib2019robust}, but our work lies between these two extremes.

\subsection{Support Knapsack Constraints}
We first consider a knapsack constraint, meaning the function $W$ has the form:
\begin{align*}
W(\mathbf{x}) &= \sum_{j\in\supp{\mathbf{x}}}\mathbf{w}_j,
\end{align*}
for some $\mathbf{w}\in\mathbb{R}^n_{>0}$.  The regularized problem \eqref{eq:regularized_case} in this case is:
\begin{align*}
\minimize{\mathbf{x}\in\mathbb{R}^n_{\geq 0}}&~f(\mathbf{x}) + g(\supp{\mathbf{x}}) + \mu \sum_{j\in\supp{\mathbf{x}}}\mathbf{w}_j.
\end{align*}
Because $W$ is a set function in this case, the relaxed problem \eqref{eq:lattice_opt_prob_relax} becomes:
\begin{align}
\minimize{A\in 2^{[n]}}~g(A) + H(A) + \sum_{j\in A}\psi_j(\mu),\label{eq:regularized_support_knapsack}
\end{align}
where we have defined $\psi_j(\mu) = \mu\mathbf{w}_j$ for each $j=1,2,...,n$.  Because $\mathbf{w}_j>0$ for all $j$, these functions are strictly increasing, and we have a problem in the form \eqref{eq:family_submodular}.  By Theorem \ref{thm:bach_convex_submod}, we can solve the convex optimization problem:
\begin{align*}
\minimize{\mathbf{u}\in\mathbb{R}_{\geq \epsilon}}~g_L(\mathbf{u}) + H_L(\mathbf{u}) + \frac{1}{2}\sum_{j=1}^n\mathbf{w}_j\mathbf{u}_j^2,
\end{align*}
then appropriately threshold the solution to recover the solution to \eqref{eq:regularized_support_knapsack} for all possible values of $\mu\in\mathbb{R}_{\geq\epsilon}$.  Because $\psi_j$ is finite and strictly increasing on all of $\mathbb{R}$, we can simply select $\epsilon = 0$.

Given the solutions to the regularized problem $A^\mu$ specified by Theorem \ref{thm:bach_convex_submod}, we select the set $A^\mu$ with smallest $\mu\in\mathbb{R}$ such that the constraint $W(\mathbf{x}) \leq B$ is satisfied.  Note however, that we only recover the solution for \emph{any} given $B\in\mathbb{R}_{\geq 0}$ if the elements of $\mathbf{u}^*$ are unique \cite{bach2013learning}.  Otherwise, we only recover the solutions for a few particular values of $B$.  If these elements are unique, however, we can use the result of Theorem \ref{thm:main_result} to compute the minimizer in the original optimization problem over $\mathbb{R}^n_{\geq 0}$.  Moreover, by the same argument as in \cite{nagano2011size}, this solution corresponds to the solution of the original constrained problem.

\subsection{Continuous Budget Constraints}
As shown above, the Lov\`asz extension lets us handle problems with discrete budget constraints, so a natural next step is to consider continuous budget constraints, meaning continuous functions $W:\mathbb{R}^n_{\geq 0}\rightarrow\mathbb{R}$, such that:
\begin{align*}
W(\mathbf{x}) &= \sum_{i=1}^nW_i(\mathbf{x}_i),
\end{align*}
with each $W_i:\mathbb{R}_{\geq 0}\rightarrow\mathbb{R}$ a strictly increasing function. With this particular $W$, the regularized optimization problem \eqref{eq:regularized_case} with Lagrange multiplier $\mu\in\mathbb{R}_{\geq 0}$ becomes:
\begin{align*}
\minimize{\mathbf{x}\in\mathbb{R}^n_{\geq 0}}~f(\mathbf{x}) + g(\supp{\mathbf{x}}) + \mu\sum_{i=1}^nW_i(\mathbf{x}_i).
\end{align*}
To recover the problem form \eqref{eq:family_submodular} specified by Theorem \ref{thm:bach_convex_submod}, we further assume that $f:\mathbb{R}^n_{\geq 0}\rightarrow\mathbb{R}$ is separable, i.e., $f(\mathbf{x}) = \sum_{i=1}^nf_i(\mathbf{x}_i)$.  In this case, the relaxed optimization problem \eqref{eq:lattice_opt_prob_relax} is:
\begin{align}
\minimize{A\in 2^{[n]}}~g(A) + \sum_{i\in A} H_i(\mu),\label{eq:cont_budget_regularized}
\end{align}
where we defined $H_i : \mathbb{R}_{> 0}\rightarrow\mathbb{R}$ as the function:
\begin{align}
H_i(\mu) &= \underset{\mathbf{z}\geq 0}{\min}~f_i(\mathbf{z}) + \mu W_i(\mathbf{z}),\quad i = 1,2,...,n,\label{eq:scalar_H}
\end{align}
and assumed (without loss of generality) that $W_i(0) = f_i(0) = 0$.

To apply Theorem \ref{thm:bach_convex_submod}, we need $H_i:\mathbb{R}_{> 0}\rightarrow\mathbb{R}$ to be strictly increasing on its domain.  We verify this property in the following proposition, whose proof we detail in Appendix \ref{apdx:continuous_constraints}.

\begin{proposition}
The function $H_i:\mathbb{R}_{\geq 0}\rightarrow\mathbb{R}_{\leq 0}$ defined in \eqref{eq:scalar_H} is monotone in $\mu$ for all $i=1,2...,n$.  It is strictly increasing for all $\mu\in[0,c]$, where $c\in\mathbb{R}_{\geq 0}$ is the smallest constant such that $H_i(c) = 0$.  In addition, $H_j$ is constant and zero on the interval $[c,\infty[$.
\end{proposition}

Because the only point at which $H_i$ is not strictly increasing occurs when its value is exactly zero (implying that allowing the element $\mathbf{x}_i$ to be nonzero provides no decrease in continuous cost), the desired result from Theorem \ref{thm:bach_convex_submod} still holds with only a minor modification, the details of which we also defer to Appendix \ref{apdx:continuous_constraints}.

It then follows from Theorem \ref{thm:bach_convex_submod} that by solving the single convex optimization problem:
\begin{align}
\minimize{\mathbf{u}\in\mathbb{R}^n_{\geq 0}}~g_L(\mathbf{u}) + \sum_{i=1}^n\int_\epsilon^{\epsilon+\mathbf{u}_i}H_i(\mu)d\mu,
\end{align}
we can recover the solution to a family of regularized optimization problems \eqref{eq:cont_budget_regularized}.  As before, we select the set $A^\mu$ with the largest $\mu\in\mathbb{R}_{\geq \epsilon}$ such that the budget constraint $W(\mathbf{x})\leq B$ is satisfied.  As discussed above, we only recover the solution for \emph{all} $B\in\mathbb{R}_{\geq 0}$ if the elements of $\mathbf{u}^*$ are all unique.  Within each choice of support, simple convex duality--which we can apply when $f_i$ and $W_i$ are convex functions--guarantees the existence of a $\mu\in\mathbb{R}_{\geq 0}$ that renders the constrained problem and the regularized problem equivalent.

\section{Robust Optimization}
Joint continuous and discrete optimization problems can easily arise as sub-problems in larger contexts.  For example, in \emph{robust optimization}, we seek to solve an optimization problem while remaining resilient to worst-case problem instances.

\subsection{Motivating Example from Multiple Domain Learning}
Recent work by \cite{qian2019robust} highlighted the concept of \emph{multiple domain learning}, where a single machine learning model is trained on sets of data from $K$ different domains.  By training against worst-case distributions of the data in these domains, they show that the resulting machine learning model often achieves lower generalization and worst-case testing errors.

In particular, let the training data for a learning model be $S = \{S_1,S_2,...,S_K\}$ with $S_i$ the data from domain $i$.  We also let $f_i : W\rightarrow\mathbb{R}$ for $i=1,2,...,K$ be the empirical risk of the model on the data from each domain $i$, given parameters in some convex subset $W\subseteq\mathbb{R}^n$.  The proposed robust optimization problem is then:
\begin{align*}
\minimize{\mathbf{w}\in W}~\underset{\mathbf{p}\in C}{\max}~\sum_{i=1}^K\mathbf{p}_if_i(\mathbf{w}),
\end{align*}
with $C = \{\mathbf{p}\in \mathbb{R}^K_{\geq 0}\mid \sum_{i=1}^K\mathbf{p}_i \leq 1\}$, the simplex.  If we additionally reward the use of data from domain $i$ (or equivalently, penalize the worst-case distribution of data for including domain $i$), then we form the robust continuous and discrete optimization problem:
\begin{align*}
\minimize{\mathbf{w}\in W}~\underset{\mathbf{p}\in C}{\max}~\sum_{i=1}^K\mathbf{p}_if_i(\mathbf{w}) - g(\supp{\mathbf{p}}),
\end{align*}
with $g: 2^{[K]}\rightarrow\mathbb{R}$ a monotone submodular set function. By considering a penalty on the set of nonzero entries of the worst-case distribution, we encode some prioritization of which domains are more or less relevant to us in our application.  Then by Theorem \ref{thm:bach_convex_submod}, we can solve the inner maximization problem (with an appropriate change of signs) by adding a Lagrange multiplier $\mu$ and solving a related convex problem.

\subsection{General Results}
More generally, robust optimization problems can often be expressed as a min-max saddle point optimization problem of a function $q:\mathcal{X}\times\mathcal{Y}\rightarrow\mathbb{R}$:
\begin{align}
\maximize{\mathbf{x}\in\mathcal{X}}\underset{\mathbf{y}\in \mathcal{Y}}{\min}~q(\mathbf{x},\mathbf{y}). \label{eq:saddle_point_prob}
\end{align}
This problem is interpreted as maximizing the function $q(\mathbf{x},\mathbf{y})$ with respect to our available parameters $\mathbf{x}\in\mathcal{X}\subseteq \mathbb{R}^n$, under the worst case choice of additional problem parameters $\mathbf{y}\in \mathcal{Y}\subseteq\mathbb{R}^m$ \cite{ben2009robust}.

Given some appropriate structure for the function $q$, the min-max problem  \eqref{eq:saddle_point_prob} is surprisingly tractable.  If we define $Q:\mathcal{X}\rightarrow\mathbb{R}$ as:
\begin{align*}
Q(\mathbf{x}) &= \underset{\mathbf{y}\in \mathcal{Y}}{\min}~q(\mathbf{x},\mathbf{y}),
\end{align*}
we can express the saddle-point problem \eqref{eq:saddle_point_prob} as:
\begin{align}
\maximize{\mathbf{x}\in\mathcal{X}}~Q(\mathbf{x}).\label{eq:Q_opt}
\end{align}
If the function $q(\mathbf{x},\mathbf{y})$ is concave in $\mathbf{x}$ for any fixed $\mathbf{y}\in \mathcal{Y}$, then the function $Q$ is also concave in $\mathbf{x}$ \cite{borwein2000convex}. Moreover, we can compute a subgradient of $Q$ at any $\mathbf{x}_0\in\mathcal{X}$ as:
\begin{gather*}
\nabla_{\mathbf{x}}Q(\mathbf{x}_0) = \nabla_\mathbf{x} q(\mathbf{x}_0,\mathbf{y}^*), \\
\mathbf{y}^*\in\argmin{\mathbf{y}\in\mathcal{Y}}~q(\mathbf{x}_0,\mathbf{y}).
\end{gather*}
In other words, efficiently solving the minimization problem defining $Q$ for an $\mathbf{x}_0\in\mathcal{X}$ also gives a subgradient of $Q$.  Because $Q$ is concave in $\mathbf{x}$, even a straightforward algorithm such as projected subgradient ascent in the problem \eqref{eq:Q_opt} will converge to a global optimum.

In this work, we showed that minimization problems in the form of \eqref{eq:cont_discrete_opt_problem} with functions satisfying Assumptions 1-3 can be solved efficiently.  Suppose then, that the function $q:\mathcal{X}\times\mathcal{Y}$ is of the form:
\begin{align*}
q(\mathbf{x},\mathbf{y}) &= f(\mathbf{x},\mathbf{y}) + g(\eta(\mathbf{y}))
\end{align*}
with $f:\mathcal{X}\times \mathcal{Y}\rightarrow\mathbb{R}$ concave in $\mathbf{x}$ for any fixed $\mathbf{y}$ and also convex and submodular on $\mathcal{Y}\subseteq \mathbb{R}^n_{\geq 0}$ in $\mathbf{y}$ for any fixed $\mathbf{x}$.  If $\eta:\mathcal{Y}\rightarrow\mathcal{L}$ satisfies Assumption 3, $g:\mathcal{L}\rightarrow\mathbb{R}$ is monotone and submodular, and we assume the set of $\mathbf{y}\in\mathcal{Y}$ such that $\eta(\mathbf{y})\leqtwo \ell$ is a convex subset for any $\ell\in\mathcal{L}$, then the robust optimization problem \eqref{eq:saddle_point_prob} becomes:
\begin{align}
\maximize{\mathbf{x}\in\mathbb{R}^n}~\underset{\mathbf{y}\in \lattwo}{\min}~f(\mathbf{x},\mathbf{y}) + g(\eta(\mathbf{y})).\label{eq:max_min_prob}
\end{align}

For a given $\mathbf{x}_0\in\mathbb{R}^n$, we view the selection of $\mathbf{y}\in \mathcal{Y}$ as a worst-case, or ``adversarial'' choice of parameters for the function $f$.  The penalty on $\eta(\mathbf{y})$ suggests that the adversarial parameters are selected while considering some preferred structure, such as sparsity.  Submodularity here, implies that this adversary pays diminishing prices as it increases the number of parameters it uses. 

In addition, $Q$ becomes:
\begin{align*}
Q(\mathbf{x}) &= \underset{\mathbf{y}\in \lattwo}{\min}~f(\mathbf{x},\mathbf{y}) + g(\eta(\mathbf{y})),
\end{align*}
which is still the minimum of a family of concave functions, and therefore amenable to subgradient ascent methods as discussed above.  A subgradient of $Q$ can easily be computed as:
\begin{gather*}
\nabla_{\mathbf{x}} Q(\mathbf{x}_0) = \nabla_{\mathbf{x}}q(\mathbf{x}_0,\mathbf{y}^*) = \nabla_\mathbf{x} f(\mathbf{x}_0,\mathbf{y}^*), \\
\mathbf{y}^* \in \argmin{\mathbf{y}\in\lattwo}~f(\mathbf{x}_0,\mathbf{y}) + g(\eta(\mathbf{y})).
\end{gather*}
We collect these ideas into the following theorem.

\begin{theorem}\label{thm:robust_opt}
Consider the robust optimization problem \eqref{eq:max_min_prob}.  Assume $f:\latone\times\lattwo\rightarrow\mathbb{R}$ is concave in $\mathbf{x}\in\latone$ for any fixed $\mathbf{y}\in\lattwo$, and also convex and submodular in $\mathbf{y}\in\lattwo$ for any fixed $\mathbf{x}\in\latone$.  Let $\eta:\lattwo\rightarrow\mathcal{L}$ satisfy Assumption 3, $g:\mathcal{L}\rightarrow\mathbb{R}$ be a monotone submodular function and assume that for a given $\mathbf{\ell}\in\mathcal{L}$, the set of $\mathbf{y}\in\lattwo$ such that $\eta(\mathbf{y})\leqtwo \mathbf{\ell}$ is a convex subset of $\lattwo$.  Moreover, let $\lattwo$ be a finitely presentable distributive lattice.  For any $\epsilon \in\mathbb{R}_{>0}$, let $T\in\mathbb{Z}_{>0}$ be of order $O(\frac{1}{\epsilon^2})$, meaning as $T$ tends to infinity, there exists a constant $M\in\mathbb{R}_{>0}$ such that $T \leq \frac{M}{\epsilon^2}$.  Then $T$ iterations of projected subgradient ascent using step lengths $\eta_i = \frac{1}{\sqrt{T}}$ produces, in polynomial time, iterates $\mathbf{x}^{(i)}\in\latone$ for $i = 1,2,...,T$ such that $\frac{1}{T}\sum_{i=1}^TQ(\mathbf{x}^{(i)}) \leq Q(\mathbf{x}^*) + \epsilon$.
\end{theorem}
The computational complexity of this approach may be high, as projected subgradient ascent can be slow in practice.  However, each sub-problem instance involves a mixed continuous and discrete optimization problem, so this complexity is warranted.

\section{Relaxing Submodularity}\label{sec:quadratic_lifting}
For the results of Theorem \ref{thm:main_result} and therefore Corollary \ref{cor:polytime} and its extensions to apply, Assumptions 1-3 must be met.  There are, however, situations where these assumptions may not hold.  For example, consider again a quadratic form for $f:\mathbb{R}^n_{\geq 0}\to\mathbb{R}$:
\begin{align}\label{eq:quadratic}
f(\mathbf{x})&=\mathbf{x}^T\mathbf{Q}\mathbf{x} + \mathbf{p}^T\mathbf{x},
\end{align}
and a monotone and submodular set function $g:2^{[n]}\to\mathbb{R}$.  Then the general lattice optimization problem \eqref{eq:lattice_opt_problem} becomes:
\begin{align}\label{eq:quadratic_opt_prob}
\begin{array}{cc}
\minimize{\mathbf{x}\in\mathbb{R}^n_{\geq 0}} & \ell(\mathbf{x}) := \mathbf{x}^T\mathbf{Q}\mathbf{x} + \mathbf{p}^T\mathbf{x} + g(\supp{\mathbf{x}}).
\end{array}
\end{align}
The assumption that $f$ is submodular on $(\mathbb{R}^n_{\geq 0},\leqone)$ is equivalent to:
\begin{align*}
\frac{\partial^2 f}{\partial \mathbf{x}_i\partial\mathbf{x}_j} &= \mathbf{Q}_{ij} \leq 0, \quad \text{for all }i\neq j.
\end{align*}
Moreover, for Corollary \ref{cor:polytime} to apply, we also need the matrix $\mathbf{Q}$ to be positive semidefinite.  These two assumptions are unlikely to both be met by quadratic forms resulting from real data.

Typically, violations of submodularity are handled by suitably relaxing the definition of submodularity with an additive or multiplicative constant \cite{elenberg2018restricted,das2018approximate}.  This constant is then propagated through the particular algorithm choice, providing a similarly relaxed optimality guarantee \cite{eloptimal}.

Alternatively, our work focuses on finding exact solutions to these joint problems in an algorithm-agnostic and efficient way.  In this spirit, we show in this section how quadratic problems such as \eqref{eq:quadratic_opt_prob} can be embedded in another optimization problem satisfying Assumptions 1-3.  We then prove conditions under which the solutions to this \emph{lifted} optimization problem--which can be efficiently found, since Assumptions 1-3 are now satisfied--correspond to an exact solution of the original quadratic problem \eqref{eq:quadratic_opt_prob}.

\subsection{Lifting Non-submodular Quadratics}
Given the quadratic form for $f$ as in \eqref{eq:quadratic}, we can decompose the matrix $\mathbf{Q}\in\mathbb{R}^{n\times n}$ into its submodular and non-submodular parts additively:
\begin{gather}\label{eq:q_matrix_split}
\mathbf{Q} = \mathbf{Q}^- + \mathbf{Q}^+, \\\label{eq:q_matrix_signs}
\begin{aligned}
\mathbf{Q}^-_{ij} &= \begin{cases} \mathbf{Q}_{ij}, & i = j\text{ or }\mathbf{Q}_{ij} \leq 0, \\ 0, &\text{otherwise},\end{cases} \qquad
\mathbf{Q}^+_{ij} = \begin{cases} \mathbf{Q}_{ij}, & i\neq j\text{ and }\mathbf{Q}_{ij} > 0 \\ 0, & \text{ otherwise}.\end{cases}
\end{aligned}
\end{gather}

Then, we define a new, lifted quadratic function $\tilde{f}:\mathbb{R}^n_{\geq 0}\times\mathbb{R}^n_{\geq 0}\to\mathbb{R}$ as:
\begin{align}\label{eq:lifted_f}
\tilde{f}(\mathbf{z},\mathbf{w}) &= \frac{1}{2}\begin{bmatrix}
\mathbf{z} \\ \mathbf{w}
\end{bmatrix}^T\begin{bmatrix}
\mathbf{Q}^- & \mathbf{Q}^+ \\
\mathbf{Q}^+ & \mathbf{Q}^-
\end{bmatrix}\begin{bmatrix}
\mathbf{z} \\ \mathbf{w}
\end{bmatrix} + \frac{1}{2}\begin{bmatrix}
\mathbf{q} \\ \mathbf{q}
\end{bmatrix}^T\begin{bmatrix}
\mathbf{z} \\ \mathbf{w}
\end{bmatrix}.
\end{align}
The lifted function $\tilde{f}$ also has some nice properties that we can use to our advantage.
\begin{lemma}\label{lem:lifted_f_symm}
The function $\tilde{f}:\mathbb{R}^n_{\geq 0}\times\mathbb{R}^n_{\geq 0}\to\mathbb{R}$ defined in \eqref{eq:lifted_f} is such that for all $(\mathbf{z},\mathbf{w})\in\mathbb{R}^n_{\geq 0}\times\mathbb{R}^n_{\geq 0}$:
\begin{gather}
\tilde{f}(\mathbf{z},\mathbf{w}) = \tilde{f}(\mathbf{w},\mathbf{z}),
\end{gather}
and for all $\mathbf{x}\in\mathbb{R}^n_{\geq 0}$:
\begin{gather}
\tilde{f}(\mathbf{x},\mathbf{x}) = f(\mathbf{x}).
\end{gather}
\end{lemma}

We can similarly lift the function $g:2^{[n]}\to\mathbb{R}$ to the function $\tilde{g}:2^{[n]}\times 2^{[n]}\to\mathbb{R}$, defined simply as:
\begin{align}\label{eq:lifted_g}
\tilde{g}(S,T) &= \frac{1}{2}\left(g(S) + g(T)\right).
\end{align}
The lifted function $\tilde{g}$ satisfies the same symmetry and embedding properties as the lifted function $\tilde{f}$.
\begin{lemma}\label{lem:lifted_g_symm}
The function $\tilde{g}$ defined in \eqref{eq:lifted_g} is such that for all $(S,T)\in 2^{[n]}\times 2^{[n]}$:
\begin{align}
\tilde{g}(S,T) = \tilde{g}(T,S),
\end{align}
and for all $A\in 2^{[n]}$:
\begin{align}
\tilde{g}(A,A) = g(A).
\end{align}
\end{lemma}

With the lifted functions $\tilde{f}$ and $\tilde{g}$ in hand, we define a lifted version of the original quadratic optimization problem \eqref{eq:quadratic_opt_prob}:
\begin{align}
\minimize{(\mathbf{z},\mathbf{w})\in\mathbb{R}^n_{\geq 0}\times \mathbb{R}^n_{\geq 0}} ~ \tilde{\ell}(\mathbf{z},\mathbf{w}) := \tilde{f}(\mathbf{z},\mathbf{w}) + \tilde{g}\left(\supp{\mathbf{z}},\supp{\mathbf{w}}\right).\label{eq:lifted_quadratic_problem}
\end{align}
If we were to solve this lifted problem and find a solution on the diagonal, i.e., a solution $(\mathbf{z}^*,\mathbf{w}^*)$ such that $\mathbf{z}^*=\mathbf{w}^*$, we immediately recover the solution to the original quadratic problem \eqref{eq:quadratic_opt_prob}.

\begin{lemma}\label{lem:lucky}
If the solution to the lifted problem \eqref{eq:lifted_quadratic_problem}, denoted $(\mathbf{z}^*,\mathbf{w}^*)\in\mathbb{R}^n_{\geq 0}\times\mathbb{R}^n_{\geq 0}$ is such that $\mathbf{z}^* = \mathbf{w}^*$, then the point $\mathbf{x}^* = \mathbf{z}^*=\mathbf{w}^*$ is an optimal solution to the original quadratic problem \eqref{eq:quadratic_opt_prob}.
\end{lemma}
\begin{proof}
By Lemmas \ref{lem:lifted_f_symm} and \ref{lem:lifted_g_symm}, we know that:
\begin{align*}
\tilde{\ell}(\mathbf{z}^*,\mathbf{w}^*) &= \ell(\mathbf{z}^*) = \ell(\mathbf{w}^*).
\end{align*}
Further, by the optimality of $(\mathbf{z}^*,\mathbf{w}^*)$ and by shrinking the feasible set, we have:
\begin{align*}
\tilde{\ell}(\mathbf{z}^*,\mathbf{w}^*) = \ell(\mathbf{z}^*) \leq \underset{\mathbf{z},\mathbf{w}\in\Rno}{\min}~\tilde{\ell}(\mathbf{z},\mathbf{w}) &\leq \underset{\substack{\mathbf{z},\mathbf{w}\in\Rno \\ \mathbf{z}=\mathbf{w}}}{\min}~\tilde{\ell}(\mathbf{z},\mathbf{w}) = \underset{\mathbf{x}\in\Rno}{\min}~\ell(\mathbf{x}).
\end{align*}
Therefore, the points $\mathbf{z}^*$ and $\mathbf{w}^*$ are also minimizers of the original problem \eqref{eq:quadratic_opt_prob}.
\end{proof}
By Lemma \ref{lem:lucky}, the solution to our initial quadratic problem is embedded in the new lifted problem \eqref{eq:lifted_quadratic_problem}.  To use this result, however, we need two key ingredients: the ability to solve the lifted problem exactly and efficiently, and a way to easily produce solutions on the diagonal.

\subsection{Efficiently solving the lifted problem}
The lifted quadratic problem \eqref{eq:lifted_quadratic_problem} has a nearly identical form to the original problem \eqref{eq:quadratic_opt_prob}, but now satisfies Assumptions 1-3, as we prove next.  As a result, we can use the approach outlined in Section \ref{sec:polytime} to solve the lifted problem.

To discuss Assumption 1 and submodularity, we define a partial order and lattice on the lifted space $\Rno\times\Rno$ so that we can discuss submodularity.  In particular, we consider the partial order $\lleq$, defined as:
\begin{align}
(\mathbf{z},\mathbf{w}) \lleq (\mathbf{z}',\mathbf{w}')\qquad \Leftrightarrow \qquad \mathbf{z} \leqone \mathbf{z}'\text{ and } \mathbf{w}\geqone \mathbf{w}',
\end{align}
where $\leqone$ denotes the partial order on $\mathbb{R}^n$ previously defined in \eqref{eq:rn_order}.  In words, we order the first part of each pair of vectors in the typical fashion, but reverse the order for the second part.  This choice of partial order also defines the join and meet operations:
\begin{align}
(\mathbf{z},\mathbf{w})\ljoin (\mathbf{z}',\mathbf{w}') &= (\mathbf{z}\joinone\mathbf{z}',\mathbf{w}\meetone\mathbf{w}') \\
(\mathbf{z},\mathbf{w})\lmeet (\mathbf{z}',\mathbf{w}') &= (\mathbf{z}\meetone\mathbf{z}',\mathbf{w}\joinone\mathbf{w}'),
\end{align}
where $\joinone$ and $\meetone$ are the join and meet operations on $(\mathbb{R}^n,\leqone)$ defined in \eqref{eq:rn_join} and \eqref{eq:rn_meet}.

By construction, then, the lifted quadratic function $\tilde{f}$ is submodular on this lattice.  Moreover, since it is a quadratic form, simple conditions guarantee its convexity.  We pursue convexity here to leverage faster exact algorithms for solving the problem, rather than the more general approach for continuous submodular minimization.  Applying the continuous submodular minimization algorithm to this lifted problem while using arbitrarily fine discretization may be of future independent interest.

\begin{lemma}\label{lem:lifted_f_submodular}
The function $\tilde{f}:\mathbb{R}^n\to\mathbb{R}^n$ defined in \eqref{eq:lifted_f} is submodular on the lattice $(\mathbb{R}^n\times\mathbb{R}^n,\lleq)$.  Further, $\tilde{f}$ is convex if and only if both  $\mathbf{Q}$ and $\mathbf{Q}^+-\mathbf{Q}^-$ are positive semidefinite.
\end{lemma}
\begin{proof}
We first note that the lattice $(\mathbb{R}^n\times\mathbb{R}^n,\lleq)$ is an \emph{orthant conic lattice}, as defined by \cite{bian2017non}.  Therefore, by Proposition 2 of \cite{bian2017non}, $\tilde{f}$ is submodular on this lattice if and only if:
\begin{align}\label{eq:cross_deriv_neg}
\frac{\partial^2 \tilde{f}}{\partial\mathbf{x}_i\partial\mathbf{x}_j} \leq 0,
\end{align}
for all $i,j=1,2,...,n$ or $i,j = n+1,n+2,...,2n$ with $i\neq j$ and:
\begin{align}\label{eq:cross_deriv_pos}
\frac{\partial^2 \tilde{f}}{\partial\mathbf{x}_i\partial\mathbf{x}_j} \geq 0,
\end{align}
for all $i = 1,2,...,n$ and $j=n+1,n+2,...,2n$.  For our lifted function $\tilde{f}$, its Hessian matrix is exactly:
\begin{align*}
\frac{\partial^2 \tilde{f}}{\partial\mathbf{x}^2} &= \begin{bmatrix}\mathbf{Q}^- & \mathbf{Q}^+ \\ \mathbf{Q}^+ & \mathbf{Q}^- \end{bmatrix}.
\end{align*}
By their construction, the matrices $\mathbf{Q}^+$ and $\mathbf{Q}^-$ satisfy both \eqref{eq:cross_deriv_neg} and \eqref{eq:cross_deriv_pos}, and $\tilde{f}$ is submodular on $(\mathbb{R}^n\times\mathbb{R}^n,\lleq)$.

For convexity, we note that the Hessian matrix must be positive semidefinite.  By the matrix similarity:
\begin{align*}
\frac{1}{2}\begin{bmatrix}
\mathbf{I} & -\mathbf{I} \\
\mathbf{I} & \mathbf{I}
\end{bmatrix}\begin{bmatrix}
\mathbf{Q}^+ & \mathbf{Q}^- \\
\mathbf{Q}^- & \mathbf{Q}+
\end{bmatrix}\begin{bmatrix}
\mathbf{I} & \mathbf{I} \\
-\mathbf{I} & \mathbf{I}
\end{bmatrix} &= \begin{bmatrix}
\mathbf{Q}^+ - \mathbf{Q}^- & \mathbf{0} \\
\mathbf{0} & \mathbf{Q}^+ + \mathbf{Q}^-
\end{bmatrix},
\end{align*}
this holds only when $\mathbf{Q} = \mathbf{Q}^+ + \mathbf{Q}^-$ and $\mathbf{Q}^+ - \mathbf{Q}^-$ are positive semidefinite.
\end{proof}
Similarly, we define a lattice in the lifted discrete space $\twon\times\twon$ using the partial order $\Subset$ defined as:
\begin{align*}
(S,T) \Subset (S',T') \qquad \Leftrightarrow S\subseteq S' \text{ and } T\supseteq T'.
\end{align*}
The join and meet operations on $(2^{[n]}\times 2^{[n]},\Subset)$, denoted by $\lsjoin$ and $\lsmeet$ respectively, are:
\begin{align*}
(S,T)\lsjoin(S',T') &= (S\cup S', T\cap T') \\
(S,T)\lsmeet(S',T') &= (S\cap S', T\cup T').
\end{align*}
We can then easily establish that the lifted function $\tilde{g}$ is submodular on the lifted discrete lattice.
\begin{lemma}\label{lem:lifted_g_submodular}
If the function $g:2^{[n]}\to\mathbb{R}$ is monotone and submodular, then the lifted function $\tilde{g}$ defined in \eqref{eq:lifted_g} is submodular on the lattice $(2^{[n]}\times 2^{[n]},\Subset)$.  Moreover, it is monotone and submodular on the product lattice, $(2^{[n]}\times 2^{[n]},\subseteq)$.
\end{lemma}
\begin{proof}
Take a set $(S,T)\in 2^{[n]}\times 2^{[n]}$ and another set $(S',T')\in 2^{[n]}\times 2^{[n]}$.  Then by definition, we have:
\begin{align*}
\tilde{g}(S,T) + \tilde{g}(S',T') &= \frac{1}{2}\left(g(S) + g(T) + g(S') + g(T')\right) \\
&\geq \frac{1}{2}\left(g(S\cap S') + g(S\cup S') + g(T\cap T') + g(T\cup T')\right) \\
&= \tilde{g}\left((S,T)\lsjoin (S',T')\right) + \tilde{g}\left((S,T)\lsmeet (S',T')\right),
\end{align*}
where the inequality follows from the submodularity of $g$, with $\lsjoin$ and $\lsmeet$ the join and meet operations associated with the partial order $\Subset$ on $2^{[n]}\times 2^{[n]}$.  By grouping terms differently, we also see that $\tilde{g}$ is also monotone and submodular on the more typical product lattice $(\twon\times\twon,\subseteq)$.
\end{proof}
Because $\tilde{g}$ is monotone on the product lattice and $\tilde{h}$ is submodular on $(\Rno\times\Rno,\lleq)$, Lemma \ref{lem:minimizers} applies, and we can define the parameterized function $\tilde{h}:2^{[n]}\times 2^{[n]}\to\mathbb{R}$:
\begin{align}
\tilde{h}(S,T) &= \underset{\substack{\mathbf{z},\mathbf{w}\in\mathbb{R}^n_{\geq 0} \\ \supp{\mathbf{z}}\subseteq S \\ \supp{\mathbf{w}}\subseteq T}}{\min} \tilde{f}(\mathbf{z},\mathbf{w}),
\end{align}
and then the solution to:
\begin{align}\label{eq:lifted_parameterized_problem}
\minimize{S,T \in 2^{[n]}\times 2^{[n]}}~\tilde{g}(S,T) + \tilde{h}(S,T)
\end{align}
corresponds to a solution of the lifted problem \eqref{eq:lifted_quadratic_problem}.

Finally, note that Assumptions 1 and 3 are satisfied by $\tilde{f}$, $\tilde{g}$, the lattices $(\twon\times\twon,\Subset)$ and $(\Rno\times\Rno,\lleq)$, and the mapping $\textrm{supp}:\Rno\times\Rno\to\twon\times\twon$.  Therefore, we have the following direct corollary of Theorem \ref{thm:main_result}.

\begin{corollary}\label{cor:lifted_parameterized_submodular}
The function $\tilde{h}:2^{[n]}\times 2^{[n]}$ is submodular on the lattice $(2^{[n]}\times 2^{[n]},\Subset)$.
\end{corollary}

Finally, if the non-submodular contribution to the quadratic form is not too large, particularly if $\mathbf{Q}^+ - \mathbf{Q}^-$ is positive semidefinite, then by Lemma \ref{lem:lifted_f_submodular} $\tilde{f}$ is also convex.  Under this assumption, Corollary \ref{cor:polytime} applies, so we can solve the lifted optimization problem exactly in polynomial time.

\begin{corollary}\label{cor:polytime_lift}
Under the same assumptions as Corollary \ref{cor:polytime}, if $\mathbf{Q}$ and $\mathbf{Q}^+-\mathbf{Q}^-$ are both positive semidefinite matrices and $g:\twon\to\R$ is monotone and submodular, then the lifted quadratic optimization problem \eqref{eq:lifted_quadratic_problem} can be solved exactly in polynomial time.
\end{corollary}

\subsection{Guarantees}
Corollary \ref{cor:polytime_lift} in the previous subsection showed that a quadratic problem that does not satisfy Assumptions 1-3 can be lifted to another quadratic problem that does.  Moreover, under mild assumptions on the problem data, the lifted problem can be solved exactly in polynomial time.  The question then arises: is this lifted problem's solution useful?

Lemma \ref{lem:lucky} stated that if we are lucky enough to compute a minimizer to the lifted problem on the diagonal, then it is also necessarily a minimizer of the original quadratic problem.  If we are unlucky, however, we would like to still to construct a minimizer of the original problem using the solution we found.  The following result shows that this is indeed possible.

\begin{lemma}\label{lem:sufficient_condition}
Let $(\mathbf{z}^*,\mathbf{w}^*)\in\mathbb{R}^n_{\geq 0}\times\mathbb{R}^n_{\geq 0}$  be a solution to the lifted quadratic optimization problem \eqref{eq:lifted_quadratic_problem}.  If:
\begin{align}
(\mathbf{z}^*-\mathbf{w}^*)^T\mathbf{Q}^-(\mathbf{z}^*-\mathbf{w}^*) \leq 0,
\end{align}
then both $(\mathbf{z}^*,\mathbf{z}^*)$ and $(\mathbf{w}^*,\mathbf{w}^*)$ are also minimizers of the lifted problem.  By extension, $\mathbf{z}^*$ and $\mathbf{w}^*$ are minimizers of the original quadratic problem \eqref{eq:quadratic_opt_prob}.
\end{lemma}
\begin{proof}
By Proposition \ref{prop:symm_identity} (in the appendix), we have that:
\begin{align*}
\tilde{\ell}(\mathbf{z}^*,\mathbf{z}^*) + \tilde{\ell}(\mathbf{w}^*,\mathbf{w}^*) &= 2\tilde{\ell}(\mathbf{z}^*,\mathbf{w}^*) + (\mathbf{z}^*-\mathbf{w}^*)^T\mathbf{Q}^-(\mathbf{z}^*-\mathbf{w}^*).
\end{align*}
Re-arranging, and applying the optimality of $(\mathbf{z}^*,\mathbf{w}^*)$, it follows that:
\begin{align*}
(\mathbf{z}^*-\mathbf{w}^*)^T\mathbf{Q}^-(\mathbf{z}^*-\mathbf{w}^*) &=  \underbrace{\tilde{\ell}(\mathbf{z}^*,\mathbf{z}^*) - \tilde{\ell}(\mathbf{z}^*,\mathbf{w}^*)}_{\geq 0} + \underbrace{\tilde{\ell}(\mathbf{w}^*,\mathbf{w}^*) - \tilde{\ell}(\mathbf{z}^*,\mathbf{w}^*)}_{\geq 0} \geq 0.
\end{align*}
Next, by assumption, $(\mathbf{z}^*-\mathbf{w}^*)^T\mathbf{Q}^-(\mathbf{z}^*-\mathbf{w}^*) \leq 0$, and therefore:
\begin{align*}\tilde{\ell}(\mathbf{z}^*,\mathbf{z}^*) - \tilde{\ell}(\mathbf{z}^*,\mathbf{w}^*) +\tilde{\ell}(\mathbf{w}^*,\mathbf{w}^*) - \tilde{\ell}(\mathbf{z}^*,\mathbf{w}^*) &= 0.
\end{align*}
If we again re-arrange and apply the optimality of $(\mathbf{z}^*,\mathbf{w}^*)$, we find:
\begin{align*}
0\leq \tilde{\ell}(\mathbf{z}^*,\mathbf{z}^*)-\tilde{\ell}(\mathbf{z}^*,\mathbf{w}^*) = \tilde{\ell}(\mathbf{z}^*,\mathbf{w}^*) - \tilde{\ell}(\mathbf{w}^*,\mathbf{w}^*) \leq 0,
\end{align*}
and therefore we have:
\begin{align*}
\tilde{\ell}(\mathbf{z}^*,\mathbf{z}^*) = \tilde{\ell}(\mathbf{z}^*,\mathbf{w}^*) = \tilde{\ell}(\mathbf{w}^*,\mathbf{w}^*),
\end{align*}
and by Lemma \ref{lem:lucky} the points $\mathbf{z}^*$ and $\mathbf{w}^*$ are both minimizers of the original quadratic problem \eqref{eq:quadratic_opt_prob}.
\end{proof}
Note then that for any minimizer $(\mathbf{z}^*,\mathbf{w}^*)$ of the lifted problem \eqref{eq:lifted_quadratic_problem}, by the submodularity of $\tilde{f}$ and $\tilde{g}$ and the definition of the lattice $(\mathbb{R}^n_{\geq 0}\times\mathbb{R}^n_{\geq 0},\lleq)$, we can also construct the minimizer $(\mathbf{z}^*\joinone\mathbf{w}^*,\mathbf{z}^*\meetone\mathbf{w}^*)$ and its counterpart, $(\mathbf{z}^*\meetone\mathbf{w}^*,\mathbf{z}^*\joinone\mathbf{w}^*)$.  If \emph{any} of these minimizers satisfy the criteria of Lemma \ref{lem:sufficient_condition}, then we immediately recover an optimal solution of the original quadratic problem.

The conditions required by Lemma \ref{lem:sufficient_condition} are in fact not only sufficient, but necessary.  In particular, any two solutions that are on the diagonal must satisfy them.  We defer its proof to the appendix because of its similarity to the proof of Lemma \ref{lem:sufficient_condition}.

\begin{lemma}\label{lem:necessary_condition}
If $(\mathbf{z}^*,\mathbf{z}^*)$ and $(\mathbf{w}^*,\mathbf{w}^*)$ are minimizers of the lifted problem \eqref{eq:lifted_quadratic_problem}, then:
\begin{align*}
(\mathbf{z}^*-\mathbf{w}^*)^T\mathbf{Q}^-(\mathbf{z}^*-\mathbf{w}^*) \leq 0.
\end{align*}
\end{lemma}

Lemmas \ref{lem:sufficient_condition} and \ref{lem:necessary_condition} show that the easily verified quadratic form condition on the solutions to the lifted problem are both necessary and sufficient.  In practice, we can simply solve the lifted problem and then check if the condition holds.

What might happen if the conditions of Lemma \ref{lem:sufficient_condition} are not satisfied, but we use its suggested minimizer anyways?  It turns out that these solutions are still nearly optimal, with the distance from optimality measured using the same necessary and sufficient condition in Lemmas \ref{lem:sufficient_condition} and \ref{lem:necessary_condition}.

\begin{lemma}\label{lem:bound}
Let $\mathbf{x}^*\in\mathbb{R}^n_{\geq 0}$ be a minimizer of the original quadratic problem \eqref{eq:quadratic_opt_prob}, and $(\mathbf{z}^*,\mathbf{w}^*)\in\mathbb{R}^n_{\geq 0}\times\mathbb{R}^n_{\geq 0}$ be a minimizer of the lifted quadratic problem \eqref{eq:lifted_quadratic_problem}.  Then:
\begin{align*}
\min\{\ell(\mathbf{z}^*),\ell(\mathbf{w}^*)\} \leq \ell(\mathbf{x}^*) + (\mathbf{z}^*-\mathbf{w}^*)^T\mathbf{Q}^-(\mathbf{z}^*-\mathbf{w}^*).
\end{align*}
\end{lemma}
\begin{proof}
Again applying Proposition \ref{prop:symm_identity}, we have:
\begin{align*}
\tilde{\ell}(\mathbf{z}^*,\mathbf{z}^*) + \tilde{\ell}(\mathbf{w}^*,\mathbf{w}^*) &= 2\tilde{\ell}(\mathbf{z}^*,\mathbf{w}^*) + (\mathbf{z}^*-\mathbf{w}^*)^T\mathbf{Q}^-(\mathbf{z}^*-\mathbf{w}^*).
\end{align*}
Then, applying the optimality of $(\mathbf{z}^*,\mathbf{w}^*)$, we upper bound the right hand side:
\begin{align*}
\tilde{\ell}(\mathbf{z}^*,\mathbf{z}^*) + \tilde{\ell}(\mathbf{w}^*,\mathbf{w}^*) &= 2\tilde{\ell}(\mathbf{z}^*,\mathbf{w}^*) + (\mathbf{z}^*-\mathbf{w}^*)^T\mathbf{Q}^-(\mathbf{z}^*-\mathbf{w}^*) \\
&\leq 2\tilde{\ell}(\mathbf{x}^*,\mathbf{x}^*) + (\mathbf{z}^*-\mathbf{w}^*)^T\mathbf{Q}^-(\mathbf{z}^*-\mathbf{w}^*).
\end{align*}
If we divide by two note that the minimum is less than the average, we have:
\begin{gather*}
\tilde{\ell}(\mathbf{z}^*,\mathbf{w}^*) \leq 2\tilde{\ell}(\mathbf{x}^*,\mathbf{x}^*) + (\mathbf{z}^*-\mathbf{w}^*)^T\mathbf{Q}^-(\mathbf{z}^*-\mathbf{w}^*) \\
\Rightarrow \min\{\tilde{\ell}(\mathbf{z}^*,\mathbf{z}^*),\tilde{\ell}(\mathbf{w}^*,\mathbf{w}^*)\} \leq \tilde{\ell}(\mathbf{x}^*,\mathbf{x}^*) + \frac{1}{2}(\mathbf{z}^*-\mathbf{w}^*)^T\mathbf{Q}^-(\mathbf{z}^*-\mathbf{w}^*).
\end{gather*}
Then, by Lemmas \ref{lem:lifted_f_symm} and \ref{lem:lifted_g_symm}, this implies the result:
\begin{align*}
\min\{\tilde{\ell}(\mathbf{z}^*,\mathbf{z}^*),\tilde{\ell}(\mathbf{w}^*,\mathbf{w}^*)\} = \min\{\ell(\mathbf{z}^*),\ell(\mathbf{w}^*)\} \leq \ell(\mathbf{x}^*) + \frac{1}{2}(\mathbf{z}^*-\mathbf{w}^*)^T\mathbf{Q}^-(\mathbf{z}^*-\mathbf{w}^*).
\end{align*}
\end{proof}
This series of results suggests the following approach for quadratic problems that violate Assumption 1: lift the problem to a higher-dimensional one satisfying all the required assumptions, solve the new lifted problem, then check if the conditions for Lemma \ref{lem:sufficient_condition} are satisfied.  If so, then construct the associated minimizer of the original problem.  If the conditions are not satisfied, the value we computed immediately gives an additive bound on the suboptimality of the result.

\section{Examples and Computational Evaluation}
In this section, we illustrate the proposed theoretical results on several numerical examples involving optimization on the lattices $\mathbb{R}^n_{\geq 0}$ and $2^{[n]}$.  We compare against two state-of-the-art techniques: a direct application of the continuous submodular function minimization algorithms outlined by \cite{bach2019submodular}, and the projected subgradient descent method proposed in \cite{eloptimal}.  

The algorithms for continuous submodular function minimization operate by discretizing the domain $\mathbb{R}^n_{\geq 0}$ into $k$ discrete points in each dimension, converting the continuous optimization problem into a submodular minimization problem over a bounded integer lattice.  In our examples, we consider the domain $[0,1]^n\subseteq\mathbb{R}^n_{\geq 0}$ and set the discretization level to $k = 51$ unless otherwise specified.  The algorithms for continuous submodular function minimization then solve an equivalent convex optimization problem (defined using a generalized Lov\'asz extension for the integer lattice) using projected subgradient or Frank-Wolfe techniques. In our implementation, we use the Pairwise Frank-Wolfe algorithm to solve this convex problem, with all relevant results plotted in blue and labeled \textit{Cont Submodular}.

The projected subgradient method is known to provide approximation guarantees even in the non-submodular case \cite{eloptimal}, but as shown in Section \ref{sec:polytime}, amounts to a specific choice of algorithms in our theory.  The algorithm operates by solving an equivalent convex optimization problem--in particular, minimizing the Lov\'asz extension of $g+H$ over $[0,1]^n$--using projected subgradient descent.  To implement this approach, we use IBM's CPLEX 12.8 constrained quadratic program solver in MATLAB to evaluate the function $H$ (as expressed in \eqref{eq:H_definition}) and use Polyak's rule for updating the step size.  The relevant results are plotted in red, and labeled \textit{PGD + CPLEX} in figures.

Our approach is agnostic to the choice of convex optimization and submodular set function minimization routines, so we also use CPLEX to evaluate $H$. To highlight the utility of an algorithm-agnostic approach, we also implement an active-set method for fast non-negative quadratic programming to evaluate $H$ \cite{bro1997fast}.  For the submodular set function minimization algorithm, we use the minimum-norm point algorithm from \cite{fujishige2011submodular} as implemented in MATLAB by \cite{krause2010sfo}, coupled with the semi-gradient lattice pruning strategy proposed by \cite{iyer2013fast} which has quadratic complexity and drastically reduces the problem size.  Our results are plotted in black, and labeled \textit{MNP + CPLEX} and \textit{MNP + FNNQP} in figures.

The various methods are given identical cost functions to minimize, and are run until either convergence to suboptimality below $10^{-4}$ or a maximum of 100 iterations.  The experiments were all run on a laptop with an AMD Ryzen 9 4900HS CPU and 16GB of RAM.

\subsection{Regularized Sparse Regression}
We first examine a regularized sparse regression problem, similar in spirit to \eqref{eq:compressed_sensing}.  Consider some $\mathbf{x}\in\mathbb{R}^n_{\geq 0}$, $\mathbf{D}\in\mathbb{R}^{m\times n}$, $\mathbf{b}\in\mathbb{R}^m$, and define the function $f:\mathbb{R}^n_{\geq 0}\rightarrow\mathbb{R}$ as:
\begin{align}
f(\mathbf{x}) &= \Vert \mathbf{Dx}-\mathbf{b}\Vert_2^2.\label{eq:ls_mr_f}
\end{align}
Then define the monotone submodular set function $g:2^{[n]}\rightarrow\mathbb{R}$ as:
\begin{align}
g(A) &= \begin{cases}
\lambda\left[(n-1) + \max(A) - \min(A) + \vert A\vert\right], & A\neq \emptyset, \\
0 & A = \emptyset,
\end{cases}\label{eq:ls_mr_g}
\end{align}
with $\lambda\in\mathbb{R}_{\geq 0}$, and $\max(A)$ and $\min(A)$ denoting the largest and smallest index element, respectively, in the set of indices $A$.  This choice of $g$ in the sparse regression problem \eqref{eq:lattice_opt_problem} places a high penalty on large sets of nonzero entries in the vector $\mathbf{x}\in\mathbb{R}^n_{\geq 0}$ that are far apart in index.

We generate a series of random problem instances with $m = n$ satisfying the assumption of submodularity on $\mathbb{R}^n_{\geq 0}$ and also the convexity condition of Corollary \ref{cor:polytime}.  Let $\mathrm{chol}:\mathbb{R}^{n\times n}\rightarrow\mathbb{R}^{n\times n}$ denote a Cholesky decomposition of a positive semidefinite matrix, and construct the matrix $\mathbf{D}$ in \eqref{eq:ls_mr_f} as:
\begin{align*}
\mathbf{D} &= \mathrm{chol}\left(\frac{1}{2}(\mathbf{C} + \mathbf{C}^T) + n\mathbf{I}\right), \quad \mathbf{C}_{ij}\sim\mathrm{unif}(-1,0),\4all i,j=1,2,...,n.
\end{align*}
This construction guarantees that the function $f$ in \eqref{eq:ls_mr_f} is both convex and submodular on $\mathbb{R}^n_{\geq 0}$, satisfying the conditions for Corollary \ref{cor:polytime}.  For the parameter $\mathbf{b}\in\mathbb{R}^m$, we use the signal in the top plot of Figure \ref{fig:LS_figs}, and we set the regularization strength to $\lambda = 0.05$ so that both the functions $f$ and $g$ play nontrivial roles in the combined objective function.

We plot the results from each algorithm in Figure \ref{fig:LS_figs}.  Because the minimizer of the optimization problem is a representation of $\mathbf{b}$ using structured sparse columns of $\mathbf{D}$, we show the the reconstructed vector $\mathbf{Dx}$ produced by each algorithm in the second, third, and fourth plots of Figure \ref{fig:LS_figs}.  Because there is no reliance on discretization, both the projected subgradient descent and minimum-norm point algorithms produce a much smoother result, as expected.

In the bottom left plot of Figure \ref{fig:LS_figs}, we show the cost achieved over iterations of each algorithm.  The minimum-norm point converges almost immediately to the globally optimal cost, while the projected subgradient descent method takes longer to achieve the same cost.  In contrast, the discretization error associated with the continuous submodular function minimization approach prevents it from ever achieving the true optimal cost, by a small amount.

Finally, over a small window of problem sizes, we show the running times of each algorithm in the bottom right plot of Figure \ref{fig:LS_figs}.  Interestingly, our approach presents a compromise between the slow optimality of the projected subgradient descent method and the fast but inexact continuous submodular function minimization algorithm.  Moreover, when we take advantage of the extra problem structure to use specialized algorithms, we achieve comparable running times to the continuous submodular minimization algorithm. 
\begin{figure}
\begin{center}
\begin{subfigure}{\linewidth}
	\includegraphics[width=\linewidth]{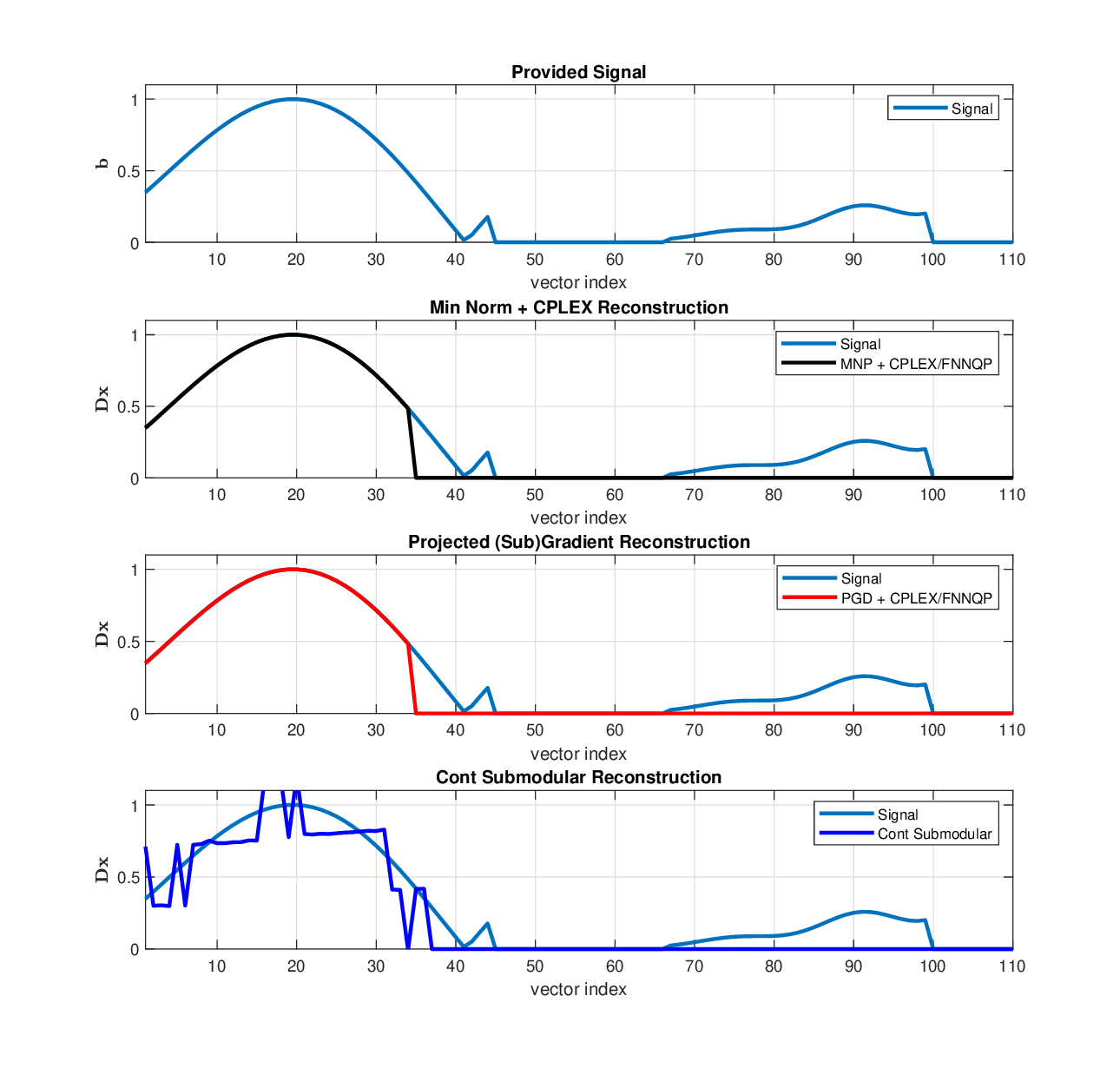}
	\vspace{-15mm}
\end{subfigure}\\
\begin{subfigure}{\linewidth}
	\includegraphics[width=\linewidth]{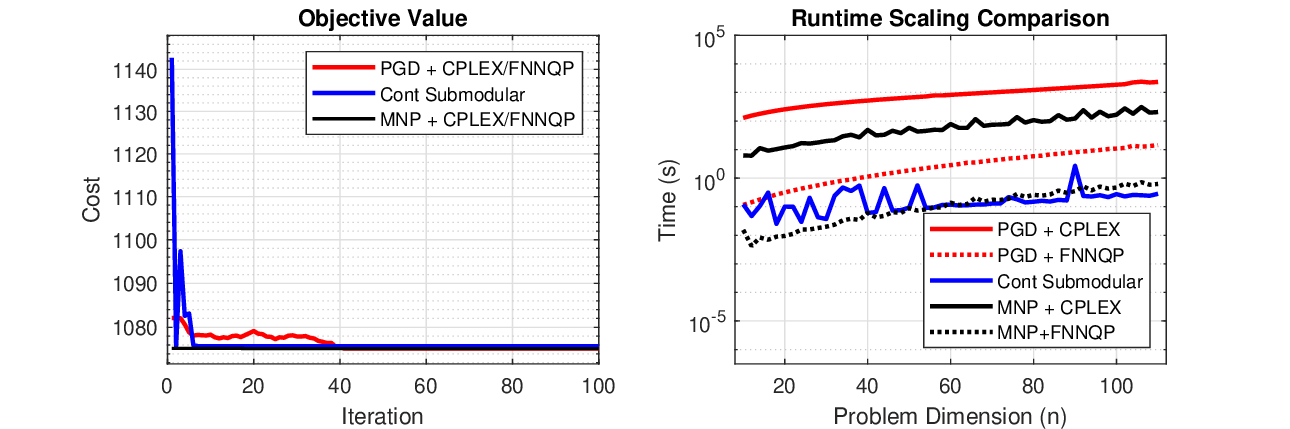}
\end{subfigure}
\end{center}
\caption{Results from the sparse regression problem simulations.  The reconstructed signal representations using columns of $\mathbf{D}$ created by each algorithm are shown in the second, third, and fourth plot. Note the solutions produced by projected subgradient and the minimum-norm point algorithm are identical. We plot the cost function value over each algorithm's iterations in the bottom left, while in the bottom right we compare the running times of the algorithms over a small window of problem dimensions.}
\label{fig:LS_figs}
\vspace{-3mm}
\end{figure}

\subsection{Signal Denoising}
We next study a simple denoising example, where we consider a signal $\mathbf{x}\in\mathbb{R}^n_{\geq 0}$, which is corrupted by some additive disturbance $\mathbf{w}\in\mathbb{R}^n$, with $\mathbf{w}\sim \mathcal{N}(0,0.1\mathbf{I})$. We would like to recover the signal $\mathbf{x}$ from the noisy measurements $\mathbf{y} = \mathbf{x} + \mathbf{w}$, under the assumption that the true signal $\mathbf{x}$ is smooth (meaning variations between adjacent entries ought to be small), and that the meaningful content arrived in a small number of contiguous sets of entries.

We can express the desire to match the noisy signal $\mathbf{y}$ with a smooth one with the convex and submodular function $f:\mathbb{R}^n\rightarrow\mathbb{R}$ defined as:
\begin{align}
f(\mathbf{x}) &= \frac{1}{2}\Vert \mathbf{x} - \mathbf{y}\Vert + \mu \sum_{i=1}^{n-1}\left(\mathbf{x}_i-\mathbf{x}_{i+1}\right)^2.\label{eq:DN_f}
\end{align}
The first term promotes matching the slightly corrupted signal, while the quadratic penalty on adjacent entries of $\mathbf{x}\in\mathbb{R}^n_{\geq 0}$ promotes smoothness.

Similarly, we can express the knowledge of a small and contiguous set of nonzero entries in the vector $\mathbf{x}$ with the monotone submodular set function $g:2^{[n]}\rightarrow\mathbb{R}$ defined by:
\begin{align}
g(A) &= \lambda\left(\vert A\vert + \mathrm{\# int}(A)\right),\label{eq:DN_g}
\end{align}
where $\lambda\in\mathbb{R}_{\geq 0}$, and the function $\mathrm{\# int}(A)$ counts the number of sets of contiguous indices in the set $A$.  This set function is smallest on subsets with a small number of entries that are adjacent in index.

For experiments, we use the signal $\mathbf{x}\in\mathbb{R}^n_{\geq 0}$ shown in the top plot of Figure \ref{fig:DN_figs}, with the noise-corrupted measurements $\mathbf{x} + \mathbf{w} = \mathbf{y}\in\mathbb{R}^n$ with an example shown in dotted orange.  We then let $\mu = 0.8$ in \eqref{eq:DN_f} and $\lambda = 0.05$ in \eqref{eq:DN_g} so that the overall problem's cost function has nontrivial contributions from both the smoothness-promoting function and the sparsity-inducing regularizer.  In this case, for the continuous submodular algorithm we discretize the compact set $[0,1]^n\subseteq\mathbb{R}^n$ into $k=51$ distinct values per index.

We show the resulting denoised signals in the second, third, and fourth plots in Figure \ref{fig:DN_figs}, with the running time comparison over a small window of problem dimensions in the bottom right.  The discretization of the domain in the continuous submodular function minimization approach produces artifacts in the reconstructed signal, whereas the result of the projected subgradient and minimum-norm point algorithms are smoother with smaller sets of nonzero entries.  We see once more that our proposed minimum-norm point algorithm poses a compromise between speed and accuracy, providing guaranteed global optimality without the high running time of projected subgradient descent.  Moreover, when we use more specialized algorithms for each sub-problem, we achieve competitive performance with the continuous submodular minimization algorithm.

We also compare the objective value achieved during the iterations of each algorithm for a single instance in the bottom left plot of Figure \ref{fig:DN_figs} with $n=100$. Again, the minimum-norm point algorithm converges almost immediately to the minimum alongside the projected subgradient method, while the continuous submodular function minimization approach's discretization error prevents it from achieving full global optimality.

\begin{figure}
\begin{center}
\begin{subfigure}{0.95\linewidth}
	\includegraphics[width=\linewidth]{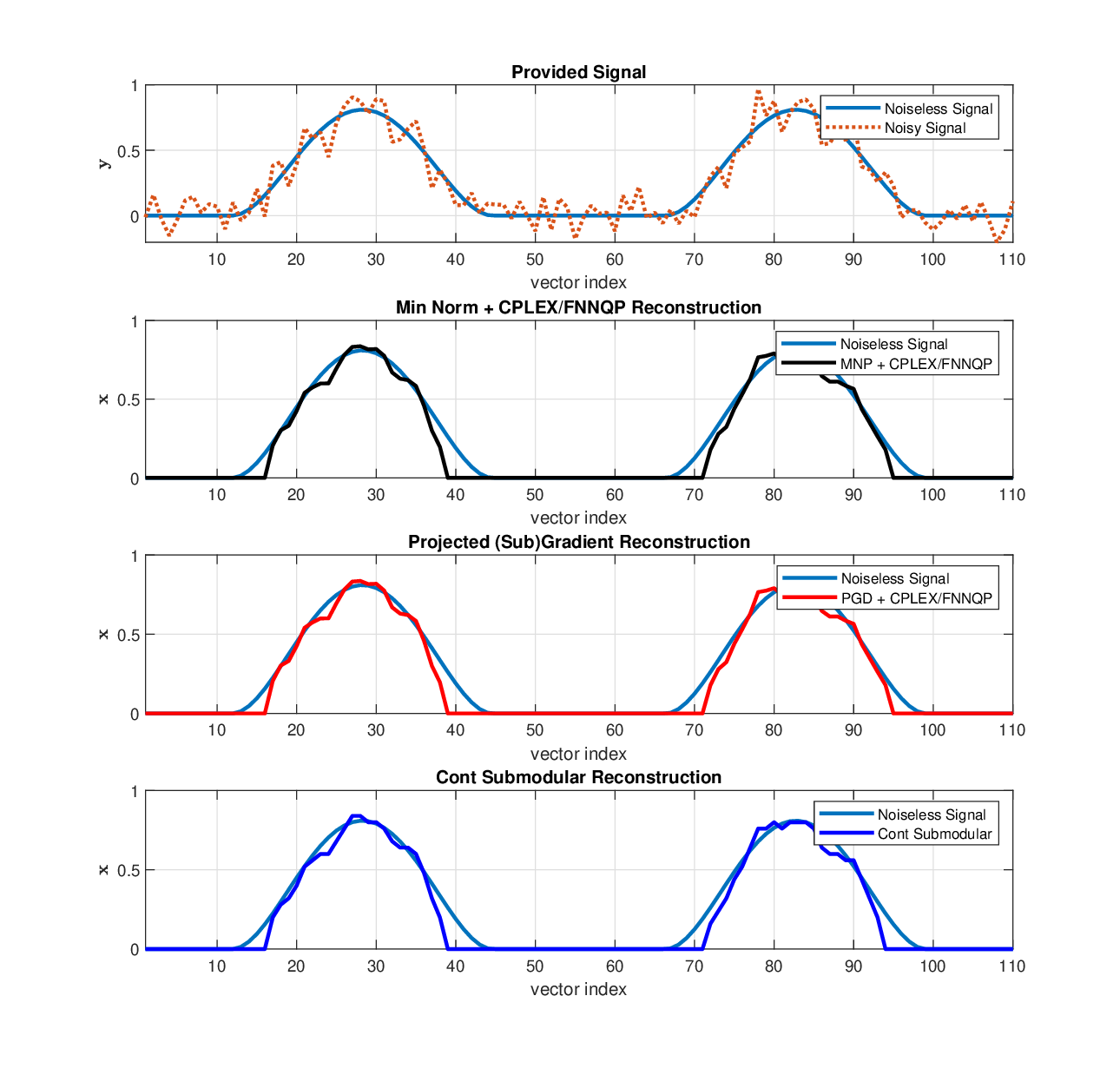}
	\vspace{-15mm}
\end{subfigure}\\
\begin{subfigure}{0.95\linewidth}
	\includegraphics[width=\linewidth]{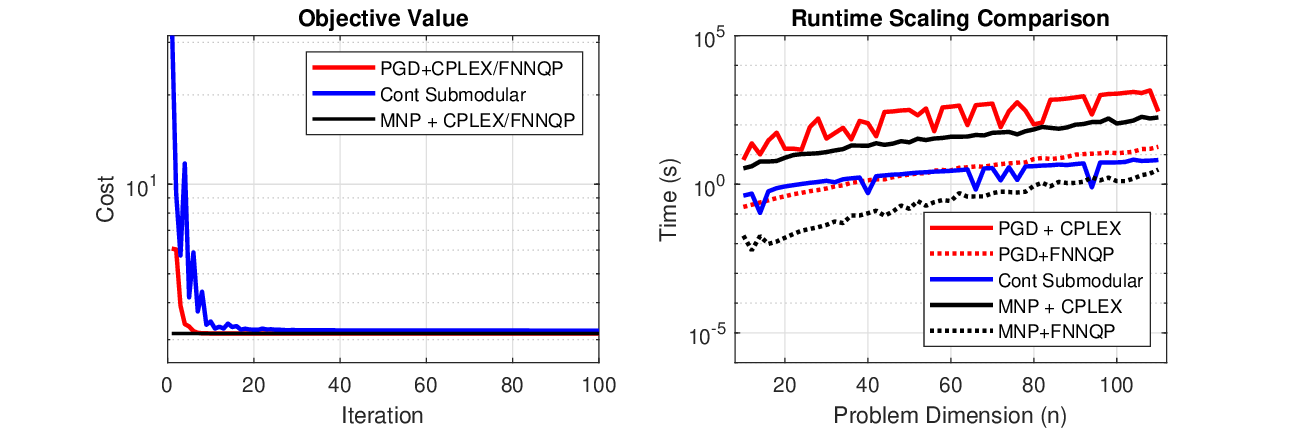}
\end{subfigure}
\end{center}
\caption{Results of the denoising problem simulations. The true signal and its noisy counterpart are shown in the top plot. The second, third, and fourth plots show the denoised signals produced by each of the three algorithms.  Note that the results from the minimum-norm point algorithm and the projected subgradient descent method are identical.  The bottom left plot shows the objective value across iterations for $n=100$, and bottom right shows the running times of each algorithm for a window of  problem dimensions.}
\label{fig:DN_figs}
\vspace{-3mm}
\end{figure}

\subsection{Price optimization with start-up costs}
In price optimization problems, we are asked to determine prices for a set of products that maximizes the expected profit while considering any inter-product demand effects caused by these prices \cite{ito2016large,ito2017optimization}.  Usually this process relies on a simple predictive model for the relationship between the price of an item and its demand, which we can easily derive with a regression technique.  Given a predictive model of the pricing-demand relationship and a characterization of our cost for each product, we want to determine the optimal pricing strategy that maximizes our profit.

Let $\mathbf{c}_i\in\mathbb{R}_{\geq 0}$ and $\mathbf{p}_i\in\mathbb{R}_{\geq 0}$ denote the cost and retail price per unit, respectively, of each item of each item $i=1,2,...,n$.  Let the function $d:\Rno\to\Rno$ be the predictive demand model, meaning that given a set of prices $\mathbf{p}$ it estimates the number of sales (or demand) of the products.  The estimated total profit of a pricing $\mathbf{p}$ can then be described by the function:
\begin{align}\label{eq:profit}
f(\mathbf{p}) &= \sum_{i=1}^n(\mathbf{p}_i-\mathbf{c}_i)d(\mathbf{p})_i.
\end{align}
Without loss of generality, we assume there is a minimum loss we are willing to accept for each item, meaning there is a lower bound $\underline{\mathbf{p}}\in\Rno$, and that if $\mathbf{p}_i=\underline{\mathbf{p}}_i$, we will not sell product $i$.

While the expression for profit \eqref{eq:profit} includes the cost of each item, it does not account for any start-up costs associated with providing them.  In particular, to provide an item, we may have to order it from a supplier and have it shipped to our facilities, paying various logistical fees to do so.  We pay these fees regardless of the \emph{quantity} of products, meaning they are a function purely of which items we choose to stock.  Moreover, in many cases these logistical costs are lumped together between items, such as when sourcing multiple products from the same supplier. 

More mathematically, assume we have $k\in\mathbb{Z}_{>0}$ groups of products with shared start-up costs, with each group represented as a subset $G_i\subseteq [n]$, each with some start-up cost $\mathbf{w}_i$.  Then the total incurred start-up costs of a subset of provided products $S$ can be expressed with a set function $g:\twon\to\R$:
\begin{align}
g(S) &= \sum_{\substack{k\in[n] \\ S\cap G_k \neq \emptyset}} \mathbf{w}_i.
\end{align}
We apply this set function to the set of products we choose to sell, $\supp{\mathbf{p}-\underline{\mathbf{p}}}\subseteq [n]$.  In this work, without loss of generality we let $\underline{\mathbf{p}} = 0$, which implies that an item priced at $\mathbf{p}_i=\underline{\mathbf{p}}_i$ earns no reward and also has no impact on the demand of the other products.  By carefully defining the demand model $d$ and costs $c$, we can enforce this property for any desired minimum price $\underline{\mathbf{p}}$.

The true underlying demand model $d$ is unknown in practice.  In a small time window, however, we can use historical data to build a local linear approximation for it, $\hat{d}:\Rno\to\Rno$:
\begin{align*}
\hat{d}(\mathbf{p}) &= \bbeta\mathbf{p} + \balpha,
\end{align*}
with $\bbeta\in\mathbb{R}^{n\times n}$ and $\balpha\in\mathbb{R}^n$.  The entries $\bbeta_{ij}$ describe the impact that the price of product $i$ has on the demand for product $j$, sometimes referred to as the \emph{elasticity of demands}\cite{ito2016large,ito2017optimization}.  Using this model, the estimated expected profit \eqref{eq:profit} is a quadratic function:
\begin{align*}
f(\mathbf{p}) &= \sum_{i=1}^n(\mathbf{p}_i-\mathbf{c}_i)\hat{d}(\mathbf{p})_i = \mathbf{p}^T\bbeta\mathbf{p} + \mathbf{p}^T(\balpha-\bbeta^T\mathbf{c}) - \mathbf{c}^T\balpha.
\end{align*}
Combining the expected profits with the start-up costs, we are faced with the optimization problem:
\begin{align}\label{eq:pricing_problem}
\begin{array}{cl}
\minimize{\mathbf{p}} & -\mathbf{p}^T\bbeta\mathbf{p} - \mathbf{p}^T(\balpha-\bbeta^T\mathbf{c}) + \mathbf{c}^T\balpha + g\left(\supp{\mathbf{p}-\underline{\mathbf{p}}}\right) \\
\text{subject to} & \mathbf{p} \geq \underline{\mathbf{p}}.
\end{array}
\end{align} 

We create this scenario with real retail sales data collected from a UK-based online retail store available in the UCI Machine Learning Repository \cite{dua2017uci,chen2012data}.  We use this data to estimate the matrix $\bbeta\in\mathbb{R}^{n\times n}$ and vector $\balpha\in\mathbb{R}^{n}$ with simple ridge regression.  To make the pricing problem \eqref{eq:pricing_problem} well-posed, we also enforce a weak diagonal dominance constraint on $\bbeta$.  In addition to making the problem well-posed, this constraint enforces the intuition that the most relevant factor in each product's demand is its own prices.

Even with a diagonal dominance constraint, the cross-terms $\bbeta_{ij}$ with $i\neq j$ can easily be either positive or negative, depending on the demand and price relationships of the products.  As a result, we cannot directly apply our parameterization method.  We can, however, use the quadratic structure of \eqref{eq:pricing_problem} and follow the results of Section \ref{sec:quadratic_lifting} to lift the pricing problem into a new quadratic problem amenable to our parameterization approach.

We compare our parameterization approach to solving \eqref{eq:pricing_problem} against the projected subgradient descent method applied directly to the original quadratic program for 100 iterations. This algorithm gives near-optimality guarantees, but explicitly computing the associated bound is NP-Hard.  Alternatively, our quadratic lifting approach gives an easily computable additive suboptimality guarantee in Lemma \ref{lem:bound} at the cost of solving a larger problem instance.  This trade-off is highlighted in the plot of running times across varying problem sizes and the achieved cost across over iterations of each algorithm for an instance of $n=20$ in Fig. \ref{fig:lifting_figs}.

We could also, in principle, use the continuous submodular minimization algorithm to solve the lifted quadratic problem.  However, this approach will still suffer inaccuracy from the discretization step, and further, runs slower than the other algorithms that take advantage of the quadratic problem structure.

\begin{figure}
\begin{center}
\begin{subfigure}{.4\textwidth}
	\includegraphics[width=1.0\linewidth]{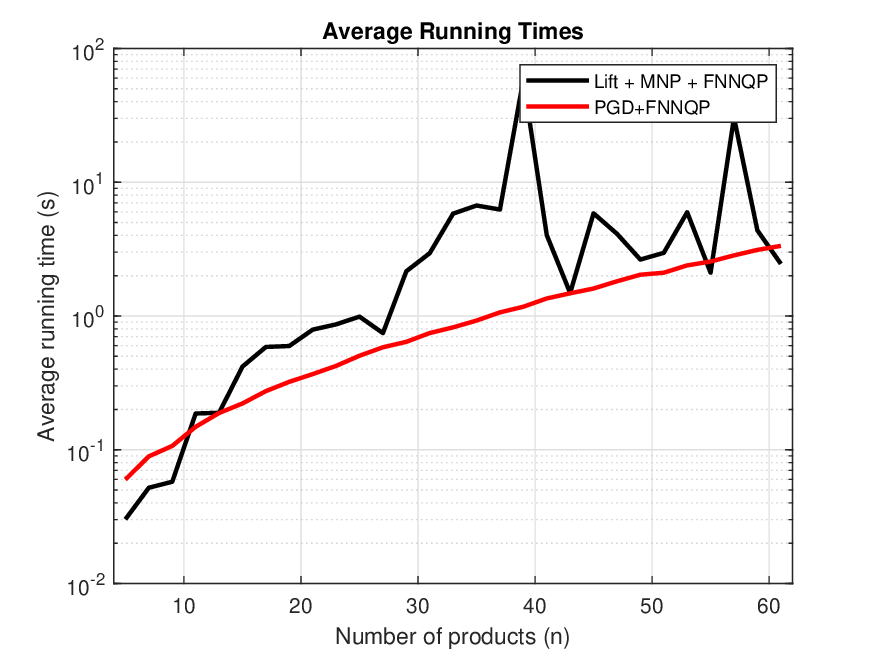}
	\vspace{-6mm}
	\label{fig:lift_runtimes}
\end{subfigure}%
\begin{subfigure}{.4\textwidth}
	\includegraphics[width=1.0\linewidth]{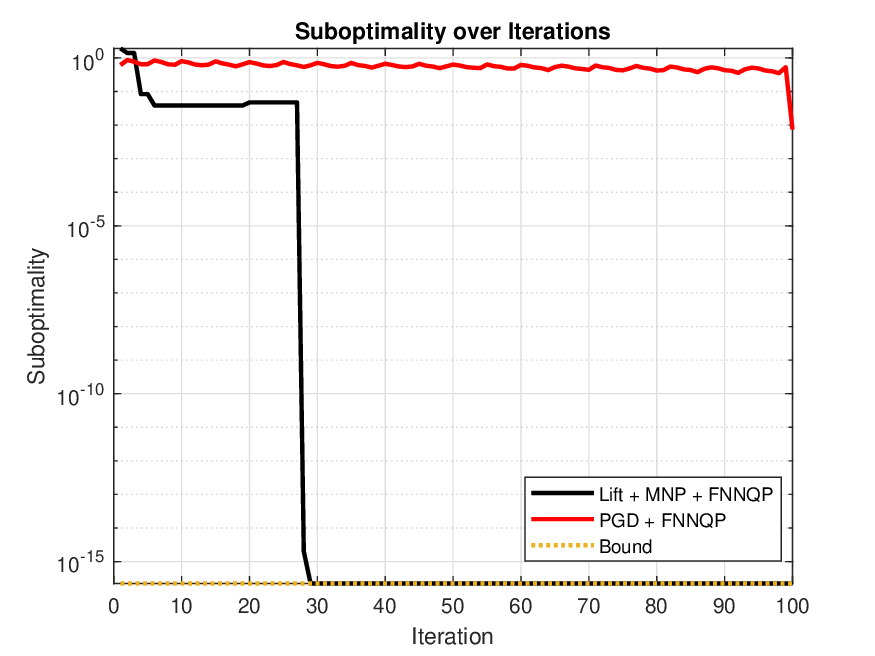}
	\vspace{-6mm}
	\label{fig:lift_costs}
\end{subfigure}
\end{center}
\caption{Results of the price optimization problem simulations.  We show the running times of each algorithm for various problem sizes (left) and the achieved cost across iterations of the algorithms for a problem of size $n=20$ (right).  The dotted line below indicates the guaranteed lower bound on the optimal solution provided by our lift.}
\label{fig:lifting_figs}
\vspace{-3mm}
\end{figure}

\subsection{Discretization Error Dependence}
In this section, we explore the relationship between the continuous submodular function minimization algorithm's discretization error and its running time.  To this end, we ran instances of the sparse regression example with the modified range function penalty, using a discretization resolution in each dimension ranging from $k=50$ to $k=400$.

The minimum cost achieved at each discretization level $k$ is shown in the left plot of Figure \ref{fig:k_comp}.  Similarly, the associated running times of the algorithm are shown in the right-hand plot of Figure \ref{fig:k_comp}.  Interestingly, near the value of $k = 250$, the achieved cost becomes effectively optimal, but the running time increases by an order of magnitude.

\begin{figure}
\begin{center}
\begin{subfigure}{.4\textwidth}
	\includegraphics[width=1.0\linewidth]{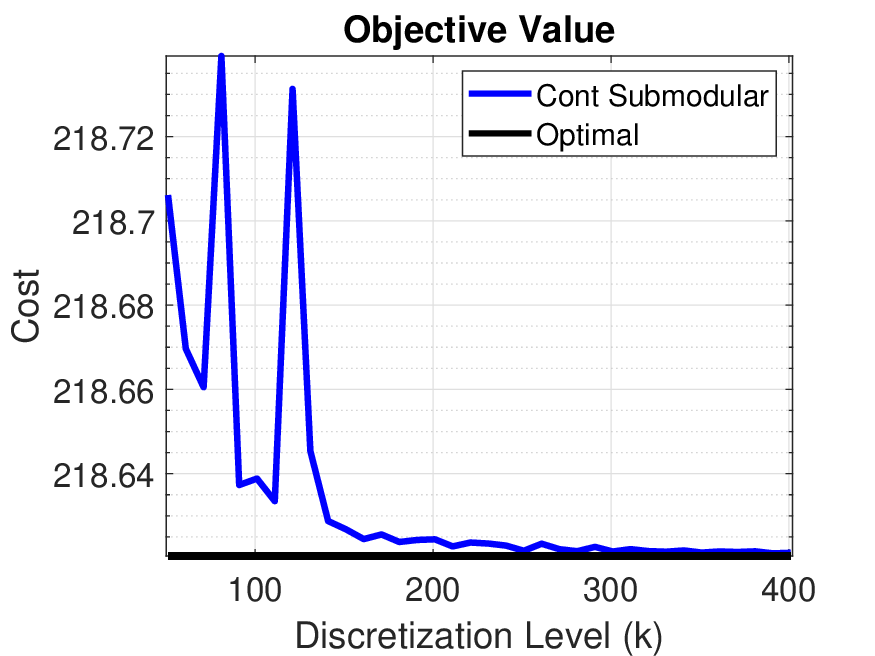}
	\vspace{-6mm}
	\label{fig:k_vs_optimality}
\end{subfigure}%
\begin{subfigure}{.4\textwidth}
	\includegraphics[width=1.0\linewidth]{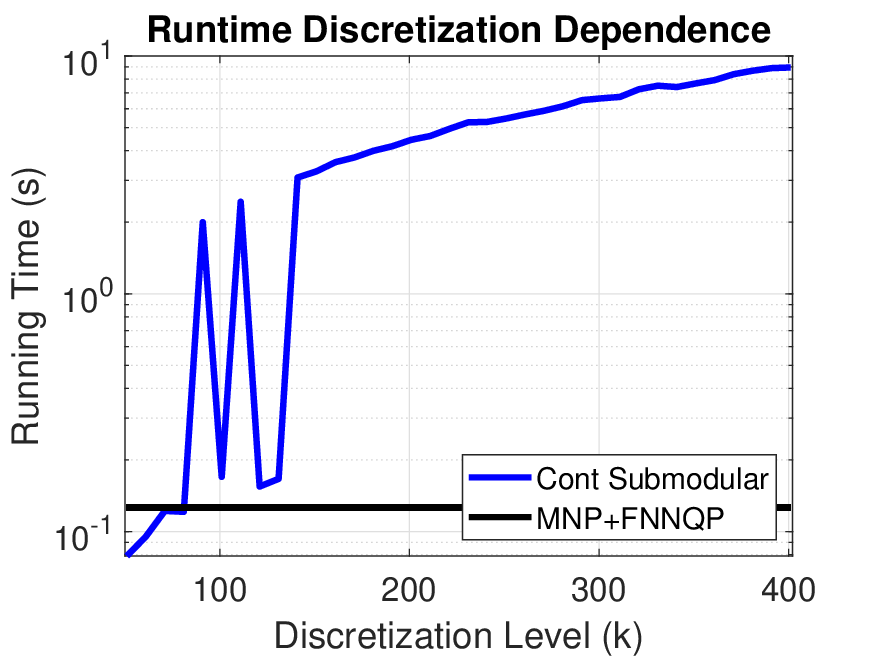}
	\vspace{-6mm}
	\label{fig:k_vs_runtime}
\end{subfigure}
\end{center}
\caption{Results highlighting the role of the discretization resolution $k$ on the continuous submodular algorithm's optimality (left) and running times (right) in an instance of the sparse regression problem with $n=100$.}
\label{fig:k_comp}
\vspace{-3mm}
\end{figure}

To give a coarse estimate on the origin of higher running times for projected subgradient descent and the minimum-norm point algorithms, we note that the computational cost of each iteration is dominated by the cost of computing the Lov\`asz extension of $H$.  This computation has time complexity $O(n\log n + n EO)$, where $EO$ is the complexity of evaluating $H$.  If $H$ is evaluated through convex optimization, many generic interior-point methods have time  complexity that is approximately $EO = O(n^3)$.  Therefore, each iteration of the minimum-norm point algorithm and the projected subgradient descent algorithm might have complexity on the order of $O(n\log n + n^4)$.  When using the fast non-negative quadratic programming algorithm, however, each evaluation operation is typically much lower than the generic $O(n^3)$.  Moreover, the lattice reduction technique of \cite{iyer2013fast} runs in approximately $O(n^2)$, and reduces the problem size drastically in many problems, as seen above.

\section{Conclusions}
In this work, we showed that model-fitting problems with structure-promoting regularizers could be expressed as optimization problems defined over two connected lattices. Using submodularity theory, we derived conditions on these functions and their domains under which we can directly solve these problems exactly and efficiently.  We focused on continuous and Boolean lattices, and derived conditions under which an agnostic combination of submodular set function minimization and convex optimization algorithms can compute the exact solution in polynomial time.

We then extended this theory to handle optimization problems with simple continuous or discrete budget constraints on the model parameters.  We did this by naively adding the constraint to the cost with a Lagrange multiplier, but then used submodular function theory to solve for all possible Lagrange multiplier values with a single convex optimization problem.  We also highlighted robust or adversarial optimization scenarios, where our exact solutions could provide subgradients to be used in globally convergent ascent methods.

Finally, we acknowledged there may be scenarios where our sufficient conditions are violated, and sought a way to weaken them without sacrificing our algorithm-agnostic approach.  To do so, we identified a class of quadratic programming problems that can be lifted to problems satisfying our conditions.  We then proved that the solutions of the lifted problem--which can then be found in polynomial time using our previously developed techniques--give provably optimal or near-optimal solutions to the original problem.  Moreover, the additive approximation bound we provide is simple to compute, unlike existing guarantees in literature that involve constants that are NP-Hard to compute.


\newpage
\newpage
\bibliographystyle{alpha}
\bibliography{cyphy_references}
\newpage
\appendix
\section{Submodularity, Lattice Morphisms, and Least Squares}\label{apdx:NNLS}
There is a massive body of work that identifies conditions under which compressed sensing problems of the form:
\begin{align}
\minimize{\mathbf{x}\in\mathbb{R}^n}~\Vert\mathbf{A}\mathbf{x}-\mathbf{b}\Vert_2^2 + \vert\supp{\mathbf{x}}\vert,\label{eq:ols_apdx}
\end{align}
for $\mathbf{A}\in\mathbb{R}^{m\times n}$ (with normalized unit norm columns, without loss of generality) and $\mathbf{b}\in\mathbb{R}^m$ can be efficiently solved by a convex relaxation of the $\ell_0$ pseudo-norm to the $\ell_1$ norm:
\begin{align*}
\minimize{\mathbf{x}\in\mathbb{R}^n}~\Vert\mathbf{A}\mathbf{x}-\mathbf{b}\Vert_2^2 + \Vert\mathbf{x}\Vert_1,
\end{align*}
with $\Vert\mathbf{x}\Vert_1 = \sum_{i=1}^n\vert\mathbf{x}_i\vert$.  The majority of these conditions rely on the matrix $\mathbf{A}$ being ``close to an isometry'', or ``nearly orthogonal''.  In this appendix, we highlight how these near-orthognality conditions on the matrix $\mathbf{A}$ can be related to the assumptions made in this work.

Interestingly, any least-squares problem in the form of \eqref{eq:ols_apdx} can be written as a least-squares problem over $\mathbb{R}^n_{\geq 0}$, by considering auxiliary variables:
\begin{align*}
\mathbf{x} &= \mathbf{x}^+ - \mathbf{x}^-, \quad \mathbf{x}^+,\mathbf{x}^-\in\mathbb{R}^n_{\geq 0}.
\end{align*}
Using these new variables, the least squares problem \eqref{eq:ols_apdx} becomes:
\begin{align*}
\minimize{\mathbf{x}^+,\mathbf{x}^-\in\mathbb{R}^n_{\geq 0}}~\left\Vert\begin{bmatrix}
\mathbf{A} & -\mathbf{A}
\end{bmatrix}\begin{bmatrix}
\mathbf{x}^+ \\
\mathbf{x}^-
\end{bmatrix}
-
\mathbf{b}\right\Vert_2^2 + \vert\supp{\mathbf{x}^+-\mathbf{x}^-}\vert.
\end{align*}
If we assume (without loss of generality) that at most one of $\mathbf{x}^+_i$ or $\mathbf{x}^-_i$ are nonzero for each $i = 1,2,...,n$, then we can equivalently write:
\begin{align*}
\minimize{\mathbf{x}^+,\mathbf{x}^-\in\mathbb{R}^n_{\geq 0}}~\begin{bmatrix}
\mathbf{x}^+ \\ \mathbf{x}^-
\end{bmatrix}^T&\begin{bmatrix}
\mathbf{A}^T\mathbf{A} & -\mathbf{A}^T\mathbf{A} \\
-\mathbf{A}^T\mathbf{A} & \mathbf{A}^T\mathbf{A}
\end{bmatrix}\begin{bmatrix}
\mathbf{x}^+\\\mathbf{x}^-
\end{bmatrix} - 2\mathbf{b}^T\begin{bmatrix}
\mathbf{A} & -\mathbf{A}
\end{bmatrix}\begin{bmatrix}
\mathbf{x}^+ \\ \mathbf{x}^-
\end{bmatrix}\\
&\quad+\vert\supp{\mathbf{x}^+}\vert + \vert\supp{\mathbf{x}^-}\vert.
\end{align*}
In this lifted problem, Assumption 1 states that the cost function must be submodular on $\mathbb{R}^n_{\geq 0}\times\mathbb{R}^n_{\geq 0}$.  For our lifted problem's cost function, this assumption is equivalent to the condition:
\begin{align*}
\left(\mathbf{A}^T\mathbf{A}\right)_{ij}&\leq 0,\quad \4all i\neq j \\
-\left(\mathbf{A}^T\mathbf{A}\right)_{ij}&\leq 0,\quad\4all i,j.
\end{align*}
This set of conditions in turn implies that $\left(\mathbf{A}^T\mathbf{A}\right)_{ii} \geq 0$ for all $i$, which is always satisfied, but also that $\left(\mathbf{A}^T\mathbf{A}\right)_{ij} = 0$ for all $i\neq j$.

By this analysis, any arbitrary least-squares problem with a monotone subset penalty can be converted to a nonnegative least-squares problem satisfying Assumptions 1-3 and the required convexity for Theorem \ref{thm:main_result} if $\mathbf{A}$ is orthogonal.  The nearness of the matrix $\mathbf{A}$ to satisfying this condition is often measured with the notion of its \emph{coherence}:
\begin{align*}
\underset{i\neq j}{\max} \left(\mathbf{A}^T\mathbf{A}\right)_{ij},
\end{align*}
which is commonly used to identify well-structured instances of least-squares problems \cite{rauhut2010compressive}.

\section{Continuous Budget Constraints}\label{apdx:continuous_constraints}
In this appendix, we prove the relevant results for continuous budget constraints.  We let $f_i:\mathbb{R}_{\geq 0}\rightarrow\mathbb{R}$ and $W_i:\mathbb{R}_{\geq 0}\rightarrow\mathbb{R}$ be continuous functions such that $f_i(0) = W_i(0) = 0$ for all $i=1,2,...,n$.  We further assume that each $W_i$ is strictly increasing for each $i$.  Then define the function $H_i:\mathbb{R}_{\geq 0}\rightarrow\mathbb{R}_{\leq 0}$:
\begin{align}
H_i(\alpha) &= \underset{\mathbf{z}\geq 0}{\min}~f_i(\mathbf{z}) + \alpha W_i(\mathbf{z}).\label{eq:h_i_scalar}
\end{align}
We first note that $H_i$ is monotone in $\alpha$.
\begin{proposition}\label{prop:increasing}
The function $H_i:\mathbb{R}_{\geq 0}\rightarrow\mathbb{R}_{\leq 0}$ is monotone in $\alpha$ for all $i=1,2,...,n$.  It is strictly increasing for all $\alpha\in[0,c]$, where $c\in\mathbb{R}_{\geq 0}$ is the smallest constant such that $H_i(c) = 0$.  Additionally, $H_i$ is constant and zero on the interval $[c,\infty[$.
\end{proposition}
\begin{proof}
Consider $\alpha,\beta\in\mathbb{R}_{\geq 0}$, with $\alpha \leq\beta$, and define the points $\mathbf{z}^\alpha \in\mathbb{R}_{\geq 0}$ and $\mathbf{z}^\beta\in\mathbb{R}_{\geq 0}$ as:
\begin{align*}
\mathbf{z}^\alpha &\in\argmin{\mathbf{z}\geq 0}f_i(\mathbf{z})+\alpha W_i(\mathbf{z}), \\
\mathbf{z}^\beta &\in \argmin{\mathbf{z}\geq 0}f_i(\mathbf{z}) + \beta W_i(\mathbf{z}).
\end{align*}
Note that for any $\alpha\in\mathbb{R}_{\geq 0}$, because $\mathbf{z} = 0$ is a feasible point in the minimization defined in \eqref{eq:h_i_scalar}:
\begin{align*}
H_i(\alpha) &= \underset{\mathbf{z}\geq 0}{\min}~f_i(\mathbf{z}) + \alpha W_i(\mathbf{z})\\
&\leq f_i(0) + \alpha W_i(0) = 0,
\end{align*}
thus $H_i$ is bounded above by zero.  Moreover, observe that by optimality of $\mathbf{z}^\alpha$:
\begin{align*}
H_i(\alpha) &= f_i(\mathbf{z}^\alpha) + \alpha W_i(\mathbf{z}^\alpha) \leq f_i(\mathbf{z}) + \alpha W_i(\mathbf{z}),\quad \4all \mathbf{z}\geq 0.
\end{align*}
Moreover, because $W_i(0) = 0$ and $W_i$ is increasing, $W_i(\mathbf{z}) \geq 0$.  Then, because $\alpha \leq \beta$:
\begin{align*}
H_i(\alpha) &= f_i(\mathbf{z}^\alpha) + \alpha W_i(\mathbf{z}^\alpha) \\
&\leq f_i(\mathbf{z}) + \alpha W_i(\mathbf{z}) \\
&\leq f_i(\mathbf{z}) + \beta W_i(\mathbf{z}), \quad \4all \mathbf{z}\geq 0.
\end{align*}
This inequality is strict when $\alpha < \beta$ and $W_i(\mathbf{z}^\alpha) \neq 0$, or equivalently $H_i(\alpha) < 0$.  In particular, because $\mathbf{z}^\beta \geq 0$:
\begin{align*}
H_i(\alpha) \leq f_i(\mathbf{z}^\beta) + \beta W_i(\mathbf{z}^\beta) = H_i(\beta),
\end{align*}
with strict inequality when $H_i(\alpha) < 0$.  Therefore $H_i$ is monotone and strictly increasing for all $\alpha\in\mathbb{R}_{\geq 0}$ such that $H_i(\alpha) < 0$.  Because it is also bounded above by zero, monotonicity implies that once $H_i(c) = 0$ for some $c\in\mathbb{R}_{\geq 0}$, it is zero for all $\beta \geq c$.
\end{proof}
Let $g:2^{[n]}\rightarrow\mathbb{R}$ be a monotone submodular set function, and consider a family of optimization problems parameterized by $\mu\in\mathbb{R}_{\geq 0}$:
\begin{align}
\minimize{A\in 2^{[n]}}~g(A) + \sum_{i\in A}H_i(\mu).\label{eq:family_probs}
\end{align}

Given Proposition \ref{prop:increasing}, we know that $H_i(0) \leq 0$ for all $i=1,2,...,n$.  If there exists an $i\in[n]$ such that $H_i(0) = 0$, Proposition \ref{prop:increasing} further states that $H_i(\alpha)$ is also zero for all $\alpha \geq 0$.  Moreover, because $g$ is monotone, we know:
\begin{align*}
g(A) + \sum_{i\in A}H_i(\alpha) &= g(A) + \sum_{i\in A\setminus\{j\}}H_i(\alpha) \\
&\geq g(A\setminus\{j\}) + \sum_{i\in A\setminus\{j\}}H_i(\alpha).
\end{align*}
In words, because $g$ is monotone and $H_i(\alpha)$ is zero for all $\alpha$, we can always reduce the cost of a subset by removing $i$.  Equivalently, we can simply remove $i$ from the ground set of elements.

We then follow the analysis in \cite{bach2013learning}, generalizing as needed to accommodate for the non-strict monotonicity of $H_i$.

\begin{proposition}(Proposition 8.2 in \cite{bach2013learning})
Let $A^\alpha$ and $A^\beta$ be minimal (i.e., smallest in size) minimizers for \eqref{eq:family_probs} with respective parameters $\alpha$ and $\beta$, with $\alpha < \beta$.  Then $A^\beta\subseteq A^\alpha$.
\end{proposition}
\begin{proof}
By the optimality of $A^\alpha$ and $A^\beta$, we have:
\begin{align}
g(A^\alpha) + \sum_{i\in A^\alpha}H_i(\alpha) &\leq g(A^\alpha\cup A^\beta) + \sum_{i\in A^\alpha\cup A^\beta} H_i(\alpha)\label{eq:opt_a_alpha} \\
g(A^\beta) + \sum_{i\in A^\beta}H_i(\beta) &\leq g(A^\alpha\cap A^\beta) + \sum_{i\in A^\alpha\cap A^\beta} H_i(\beta).\label{eq:opt_a_beta}
\end{align}
If we sum these inequalities and apply the submodularity of $g$, we have:
\begin{align}
g(A^\alpha\cup A^\beta) +  g(A^\alpha\cap A^\beta) &+ \sum_{i\in A^\alpha\cup A^\beta} H_i(\alpha) +\sum_{i\in A^\alpha\cap A^\beta} H_i(\beta)\nonumber\\
&\geq g(A^\alpha) + g(A^\beta) + \sum_{i\in A^\alpha}H_i(\alpha) + \sum_{i\in A^\beta}H_i(\beta)\nonumber \\
&\geq g(A^\alpha\cup A^\beta) + g(A^\alpha\cap A^\beta) + \sum_{i\in A^\alpha}H_i(\alpha) + \sum_{i\in A^\beta}H_i(\beta).\label{eq:sub_a_alpha_beta}
\end{align}
Subtracting equations \eqref{eq:opt_a_alpha} and \eqref{eq:opt_a_beta} from \eqref{eq:sub_a_alpha_beta}, we have:
\begin{gather}
\sum_{i\in A^\alpha\cup A^\beta} H_i(\alpha) +\sum_{i\in A^\alpha\cap A^\beta} H_i(\beta) \geq  \sum_{i\in A^\alpha}H_i(\alpha) + \sum_{i\in A^\beta}H_i(\beta) \nonumber\\
\Rightarrow \sum_{i \in A^\beta \setminus A^\alpha}\left[H_i(\beta) - H_i(\alpha)\right] \leq 0.\label{eq:diff_terms}
\end{gather}
By Proposition \ref{prop:increasing}, as $\alpha < \beta$, each $H_i(\beta)-H_i(\alpha)$ in the summation \eqref{eq:diff_terms} is strictly positive, or $H_i(\alpha) = H_i(\beta) = 0$.  But if $H_i(\alpha) = H_i(\beta) = 0$, as $g$ is monotone, we may remove $i$ from both $A^\alpha$ and $A^\beta$ and decrease the cost in \eqref{eq:family_probs}, contradicting the minimality of $A^\alpha$ and $A^\beta$.

By this argument, the left-hand side of inequality \eqref{eq:diff_terms} is the sum of strictly positive terms.  However, it is bounded above by zero, so it must therefore be the empty summation, i.e., $A^\beta\setminus A^\alpha = \emptyset$, and therefore $A^\beta \subseteq A^\alpha$.
\end{proof}
We now identify a related convex optimization problem:
\begin{align}
\minimize{\mathbf{u}\in\mathbb{R}^n_{\geq 0}}~g_L(\mathbf{u}) + \sum_{i=1}^n\int_{\epsilon}^{\epsilon+\mathbf{u}_i}H_i(\alpha)d\alpha.\label{eq:apdx_convex}
\end{align}
A classical result in submodular function theory establishes that the Lov\`asz extension $g_L$ is convex if and only if $g$ is submodular \cite{lovasz1983submodular}.  Moreover, $\int_{\epsilon}^{\epsilon+\mathbf{u}_i}H_i(\alpha)d\alpha$ is convex if and only if $H_i$ is monotone in $\alpha$, which is true by Proposition \ref{prop:increasing}.  Therefore, problem \eqref{eq:apdx_convex} is a convex optimization problem.

We now establish a relationship between the parameterized family of set function minimization problems \eqref{eq:family_probs} and the convex optimization problem \eqref{eq:apdx_convex}.
\begin{proposition}(Proposition 8.3 in \cite{bach2013learning})\label{prop:sfm_to_prox}
Given the (minimal) solutions $A^\alpha$ to the set function minimization problem \eqref{eq:family_probs} for all values of the parameter $\alpha \geq \epsilon$, define the vector $\mathbf{u}^*\in\mathbb{R}^n_{\geq 0}$ defined by:
\begin{align*}
\mathbf{u}^*_i &= \sup\left(\{\alpha \in\mathbb{R}_{\geq 0}\mid i\in A^\alpha\}\right).
\end{align*}
Then the vector $\mathbf{u}^*$ is the minimizer of the convex optimization problem \eqref{eq:apdx_convex}.
\end{proposition}
\begin{proof}
For $\alpha \geq 0$ small enough (as, without loss of generality, $H_i(0) < 0$ for all $i$), we have $H_i(\alpha) < 0$ for all $i=1,2,...,n$.  Because $g$ is monotone, for this $\alpha$, the optimal $A^\alpha$ is equal to $\{1,2,...,n\}$, and thus $\mathbf{u}$ is well defined for all $i=1,2,...,n$.

For simplicity, we use the notation $\{\mathbf{u}\geq \mu\}$ to denote the set:
\begin{align*}
\{\mathbf{u}\geq \mu\} = \{i\in\{1,2,...,n\}\mid \mathbf{u}_i \geq \mu\},
\end{align*}
for any $\mathbf{u}\in\mathbb{R}^n$ and $\mu\in\mathbb{R}$.  Then for any $\mu\geq 0$, we have:
\begin{align}
g_L(\mathbf{u}) + \sum_{i=1}^n\int_\epsilon^{\epsilon+\mathbf{u}_i}H_i&(\mu)d\mu = g_L(\mathbf{u}+\mathbf{1}\epsilon)-\epsilon g(\{1,2,..,n\}) + \sum_{i=1}^n\int_\epsilon^{\epsilon+\mathbf{u}_i}H_i(\alpha)d\alpha\nonumber \\
&= \int_0^\infty g(\{\mathbf{u}+\mathbf{1}\epsilon \geq \mu\})d\mu + \sum_{i=1}^n\int_\epsilon^{\epsilon+\mathbf{u}_i}H_i(\alpha)d\alpha - \epsilon g(\{1,2,...,n\}) \nonumber\\
&= \int_\epsilon^\infty \left[g(\{\mathbf{u} + \mathbf{1}\epsilon \geq \mu\}) + \sum_{i=1}^n \mathbbm{1}_{\{\mathbf{u_i}+\epsilon \geq \mu\}}H_i(\mu)\right]d\mu,\label{eq:int_of_set_fn}
\end{align}
where we used the indicator function defined as:
\begin{align*}
\mathbbm{1}_{\{\mathbf{u_i}^*+\epsilon \geq \mu\}} &= \begin{cases}1, & \mathbf{u}_i^* + \epsilon \geq \mu \\
0, & \text{otherwise}.
\end{cases}
\end{align*}

In the right-hand side of \eqref{eq:int_of_set_fn}, every $\mu \geq \epsilon$ in the integral defines a set function minimization for which the optimal subset is $A^\mu$.  Because we constructed $\mathbf{u}^*$ as the minimizer to each of these optimal subsets, the value at $\mathbf{u}^*$ must be lower than all other $\mathbf{u}$, leading to the inequality:
\begin{align*}
g_L(\mathbf{u}^*) + \sum_{i=1}^n\int_\epsilon^{\epsilon+\mathbf{u}^*_i}H_i(\mu)d\mu &\leq \int_\epsilon^\infty \left[g(\{\mathbf{u} + \mathbf{1}\epsilon \geq \mu\}) + \sum_{i=1}^n \mathbbm{1}_{\{\mathbf{u_j}+\epsilon \geq \mu\}}H_j(\mu)\right]d\mu \\
&= g_L(\mathbf{u}) + \sum_{i=1}^n\int_\epsilon^{\epsilon+\mathbf{u}^*_i}H_i(\mu)d\mu,
\end{align*}
for all other $\mathbf{u}\in\mathbb{R}^n_{\geq 0}$, and therefore $\mathbf{u}^*$ is optimal for \eqref{eq:apdx_convex}.
\end{proof}
Proposition \ref{prop:sfm_to_prox} establishes the relationship between the parameterized family of optimization problems \eqref{eq:family_probs} and the convex optimization problem \eqref{eq:apdx_convex}.  We state the next theorem without proof, as it requires no special modifications for our conditions.

\begin{proposition}(Proposition 8.4 in \cite{bach2013learning})
If $\mathbf{u}^*$ is the minimizer for the convex optimization problem \eqref{eq:apdx_convex}, then for all $\mu\geq \epsilon$, the minimal minimizer of the corresponding set function minimization in \eqref{eq:family_probs} is:
\begin{align*}
A^\mu &= \{i\in\{1,2,...,n\}\mid \mathbf{u}_i^* > \mu\}.
\end{align*}
\end{proposition}

This sequence of propositions ultimately abuses the interpretation of the Lov\`asz extension as an integral, and states that optimizing over the integral itself (the convex problem) and optimizing over the integrated functions for all integration variables (the set functions) is equivalent.

A noteworthy addendum is that in the definition of $H_i$, we could equivalently perform scalar minimization over a closed subset of $\mathbb{R}_{\geq 0}$, and the analysis would still follow through.  This alteration would result in effectively ``capping'' the $H_i$ functions from below, which retains the monotonicity properties necessary for the proofs.

\section{A useful symmetry property}
The lifted quadratic cost function $\tilde{c}:\Rno\times\Rno\to\R$ satisfies a convenient property that we abuse to prove several results.  We prove it here.
\begin{proposition}\label{prop:symm_identity}
Let $\tilde{\ell}$ be defined as in \eqref{eq:lifted_quadratic_problem}.  Then for any $(\mathbf{z},\mathbf{w})\in\Rno\times\Rno$, we have:
\begin{align}
\tilde{\ell}(\mathbf{z},\mathbf{z}) + \tilde{\ell}(\mathbf{w},\mathbf{w}) &= 2\tilde{\ell}(\mathbf{z},\mathbf{w}) + (\mathbf{z}-\mathbf{w})^T\mathbf{Q}^-(\mathbf{z}-\mathbf{w}).
\end{align}
\end{proposition}
\begin{proof}
We proceed by directly computing:
\begin{align*}
\tilde{\ell}(\mathbf{z},\mathbf{z}) + \tilde{\ell}(\mathbf{w},\mathbf{w}) &=  f(\mathbf{z}) + g(\supp{\mathbf{z}}) + f(\mathbf{w}) + g(\supp{\mathbf{w}}) \\
&= \mathbf{z}^T\mathbf{Q}^+\mathbf{z} + \mathbf{z}^T\mathbf{Q}^-\mathbf{z} + \mathbf{z}^T\mathbf{p} + \mathbf{w}^T\mathbf{Q}^+\mathbf{w} + \mathbf{w}^T\mathbf{Q}^-\mathbf{w} + \mathbf{w}^T\mathbf{p} \\
&\qquad + g(\supp{\mathbf{z}}) + g(\supp{\mathbf{w}}).
\end{align*}
Then, adding and subtracting the missing cross term, we have:
\begin{align*}
\tilde{\ell}(\mathbf{z},\mathbf{z}) + \tilde{\ell}(\mathbf{w},\mathbf{w}) &= \mathbf{z}^T\mathbf{Q}^+\mathbf{z}  + \mathbf{w}^T\mathbf{Q}^+\mathbf{w}+ \mathbf{z}^T\mathbf{p} + \mathbf{w}^T\mathbf{p} + g(\supp{\mathbf{z}}) + g(\supp{\mathbf{w}}) \\
&\qquad + \mathbf{z}^T\mathbf{Q}^-\mathbf{z} + \mathbf{w}^T\mathbf{Q}^-\mathbf{w} \\
&= 2\tilde{\ell}(\mathbf{z},\mathbf{w}) + \mathbf{z}^T\mathbf{Q}^-\mathbf{z} - 2\mathbf{z}^T\mathbf{Q}^-\mathbf{w} + \mathbf{w}^T\mathbf{Q}^-\mathbf{w} \\
&= 2\tilde{\ell}(\mathbf{z},\mathbf{w}) + \left(\mathbf{z}-\mathbf{w}\right)\mathbf{Q}^-\left(\mathbf{z}-\mathbf{w}\right)
\end{align*}
\end{proof}
We also provide a proof that the condition on the minimizers of the lifted problem is not only sufficient, but necessary.
\begin{lemma}
If $(\mathbf{z}^*,\mathbf{z}^*)$ and $(\mathbf{w}^*,\mathbf{w}^*)$ are minimizers of the lifted problem \eqref{eq:lifted_quadratic_problem}, then:
\begin{align*}
(\mathbf{z}^*-\mathbf{w}^*)^T\mathbf{Q}^-(\mathbf{z}^*-\mathbf{w}^*) \leq 0.
\end{align*}
\end{lemma}
\begin{proof}
Note that by the submodularity of $\tilde{\ell}$, if $(\mathbf{z}^*,\mathbf{z}^*)$ and $(\mathbf{w}^*,\mathbf{w}^*)$ are minimizers of the lifted problem \eqref{eq:lifted_quadratic_problem}, then so are their join, $(\mathbf{z}^*\joinone\mathbf{w}^*,\mathbf{z}^*\meetone\mathbf{w}^*)$ and their meet, $(\mathbf{z}^*\meetone\mathbf{w}^*,\mathbf{z}^*\joinone\mathbf{w}^*)$.  Then, working through the proof of Lemma \ref{lem:sufficient_condition} backwards proves the result.
\end{proof}
\end{document}